\documentclass[12pt,a4paper]{article}
\usepackage{latexsym,amssymb,amsfonts,amsmath,amsthm,nccmath,float,enumitem,setspace,bm,authblk}
\usepackage[usenames,dvipsnames]{xcolor}
\usepackage[hidelinks]{hyperref}
\usepackage[margin=1.8cm]{geometry}
\usepackage{bm}
\usepackage{booktabs}
\allowdisplaybreaks[1]

\hypersetup{
	colorlinks,
	linkcolor={red!80!black},
	citecolor={blue!80!black},
	urlcolor={blue!80!black}
}

\newtheorem{thm}{Theorem}[section]

\newtheorem{pro}[thm]{Proposition}

\newtheorem{rem}[thm]{Remark}

\newtheorem{lem}[thm]{Lemma}
\newtheorem{core}[thm]{Corollary}

\newtheorem{example}{Example}[section]%

\def \leq {\leqslant}
\def \geq {\geqslant}

 \def\ZZ{\mathbb Z}

\def\mod{{\sf mod~}}  

\def\diag{{\rm diag}}

\def\exp{{\sf exp}}

\def \mod#1{{\:({\rm mod}\ #1)}}

\let\oldproofname=\proofname
\renewcommand{\proofname}{\rm\bf{\oldproofname}}

\begin{document}
	
\title{Perfect state transfer on bi-Cayley graphs over abelian groups}
\author[ ]{Shixin Wang}
\author[ ]{Tao Feng}
\affil[ ]{School of Mathematics and Statistics, Beijing Jiaotong University, Beijing 100044, P. R. China}
\affil[ ]{sxwang@bjtu.edu.cn, tfeng@bjtu.edu.cn}

\renewcommand*{\Affilfont}{\small\it}
\renewcommand\Authands{ and }
\date{}
	
\maketitle
\footnotetext{Supported by NSFC under Grants 11871095 and 12271023}

\begin{abstract}
The study of perfect state transfer on graphs has attracted a great deal of attention during the past ten years because of its applications to quantum information processing and quantum computation. Perfect state transfer is understood to be a rare phenomenon. This paper establishes necessary and sufficient conditions for a bi-Cayley graph having a perfect state transfer over any given finite abelian group. As corollaries, many known and new results are obtained on Cayley graphs having perfect state transfer over abelian groups, (generalized) dihedral groups, semi-dihedral groups and generalized quaternion groups. Especially, we give an example of a connected non-normal Cayley graph over a dihedral group having perfect state transfer between two distinct vertices, which was thought impossible.
\end{abstract}

	

\section{Introduction}

Quantum random walks, introduced by Aharonov, Davidovich and Zagury \cite{ADZ} in 1993, are an analog of classical random walks in quantum systems, and have been applied extensively in quantum information and computation (cf. \cite{Amb,Childs,CG,CGW,SKW}). There are two types of quantum random walks: discrete random walks (cf. \cite{AAKV}) and continuous random walks (cf. \cite{FG}). This paper only involves continuous quantum random walks. A nice introduction to quantum random walks can be found in \cite{Kempe}.

As a particular case of quantum random walks, perfect state transfer on a graph concerns about a transfer of a quantum state from a vertex $u$ to a vertex $v$ in a special time $t$ such that the fidelity of the state transfer is unity (cf. \cite{bose,Christandl}). The problem on determining whether perfect state transfer exists is interesting because it is a crucial element of quantum information processing on emerging quantum devices (cf. \cite{BP}).

Let $\Gamma$ be an undirected simple graph with adjacency matrix $A$. Then $A$ is a symmetric $(0,1)$-matrix and all eigenvalues of $A$ are real numbers. Write $V(\Gamma)$ as the vertex set of $\Gamma$. The {\em transfer matrix} of $\Gamma$ at the time $t$, $t\in\mathbb{R}$, is defined by the following $|V(\Gamma)|\times |V(\Gamma)|$ matrix
$$
H(t):=\exp{(-\mathrm{i}tA)}
	=\sum\limits_{s=0}^{+\infty}\frac{(-\mathrm{i}tA)^s}{s!}
	=(H_{u,v}(t))_{u,v\in V(\Gamma)},
$$
where $\mathrm{i}= \sqrt{-1}$. $H(t)$ is used to describe continuous quantum random walks on the graph $\Gamma$. For any $u, v\in V(\Gamma)$, $\Gamma$ is said to have \emph{perfect state transfer} (PST) from $u$ to $v$ at the time $t$ if $|H_{u,v}(t)| = 1$. $\Gamma$ is called \emph{periodic} at $u$ if there is a time $t\neq 0$ such that $|H_{u,u}(t)| = 1$, and $\Gamma$ is called {\em periodic} if there is a time $t\neq 0$ such that $|H_{u,u}(t)| = 1$ for every vertex $u$. Define
\begin{equation*}
	\begin{split}
		T(u,v):=\{t: \Gamma\mbox{ has PST between $u$ and $v$ at the time $t$}\}.
	\end{split}
\end{equation*}
The goal is to determine $T(u, v)$ for any pair $\{u, v\}$ of vertices.

Determining the existence of perfect state transfer on a certain graph is closely related to analyze the spectrum of its adjacency matrix. Perfect state transfer is understood to be a rare phenomenon \cite{Godsil3}. Godsil \cite{Godsil2} gave a survey on perfect state transfer from a mathematical point of view. The existence of perfect state transfer has been obtained for many particular classes of graphs, for example, complete bipartite graphs \cite{Stefanak}, Hadamard diagonalizable graphs \cite{G}, circulant graphs \cite{Basic1,Basic2}, cubelike graphs \cite{Cheung}, NEPS of complete graphs \cite{L.neps}, graphs in Johnson scheme \cite{AHM}, and more generally, graphs in association schemes \cite{Coutinho}.

Using group representations, Cao et al. examined perfect state transfer for Cayley graphs over abelian groups \cite{f1}, dihedral groups \cite{C.d1,C.d2}, and semi-dihedral groups \cite{LCWW}. Bi-Cayley graphs are a natural generalization of Cayley graphs, and have received considerable attention during the past few years (see \cite{czfz,kmmm,zf} for example). Cayley graphs over abelian groups, dihedral groups, and semi-dihedral groups are all special cases of bi-Cayley graphs over abelian groups. This paper is devoted to exploring abelian bi-Cayley graphs which have perfect state transfers.

In Section 2, we give definitions on Cayley graphs and bi-Cayley graphs and recall some basic facts on finite abelian group representations. Section 3 counts eigenvalues and eigenvectors of adjacency matrices of bi-Cayley graphs over abelian groups (see Theorem \ref{thm:eigenvalues and eigenvectors of D}), and provides an explicit expression for every entry in the transfer matrix of any abelian bi-Cayley graph (see Corollary \ref{core:H-entry}). Section 4 gives necessary and sufficient conditions for an abelian bi-Cayley graph having PST (see Theorems \ref{PSTiffG0G1}, \ref{PSTiffG0G0} and \ref{thm:periodTneqvarnothing}). As corollaries, many known results in \cite{C.d1,C.d2,LCWW,f1} can be obtained. Section 5 examines PST on Cayley graphs that are also bi-Cayley graphs. Especially, we give an example of a connected non-normal Cayley graph over a dihedral group having PST between two distinct vertices (see Example \ref{ex:non-normalCay}), which produces a counterexample of Theorem $11$ in $\cite{C.d2}$. Concluding remarks are given in Section 6.

\section{Preliminaries}
	
For the group-theoretic and the graph-theoretic terminologies not defined in this paper we refer the reader to \cite{biggs,representation-theory}.
	
\subsection{Characters of finite abelian groups}
	
Let $G$ be a finite abelian group. A {\em character} $\chi$ of $G$ is a homomorphism of $G$ into the multiplicative group of complex roots of unity, that is, $\chi:G\longrightarrow\mathbb{C}^*$ such that  $\chi(ab)=\chi(a)\chi(b)$ for all $a,b\in G$. All characters of $G$ form a group $\hat{G}$ that is isomorphic to $G$ under the operation $(\chi\psi)(a):=\chi(a)\psi(a)$, where $\chi,\psi$ are characters of $G$ and $a\in G$. We often set $\hat{G}=\{\chi_g\mid g\in G\}$. The identity of the group $\hat{G}$ is the principal character $\chi_e$ ($e$ is the identity of $G$) that maps every element of $G$ to $1$.

	\subsection{Cayley graphs}\label{sec:Cayley}
	
	Cayley graphs were introduced by Arthur Cayley in 1878 to explain the concept of
	abstract groups that are described by a set of generators. Let $G$ be a group and $H$ be a subset of $G$
	such that $H$ does not contain the identity element of $G$. The {\em Cayley digraph} $\mathrm{Cay}(G, H)$ over $G$ with respect to $H$ is a digraph with the vertex set $G$
	and the edge set $\{(x,y)\mid x,y\in G, yx^{-1}\in H\}$. A Cayley digraph $\mathrm{Cay}(G, H)$ satisfying $H^{-1}=H$ is called a {\em Cayley graph}. Equivalently, a Cayley graph is a graph $(V, E)$ which admits an automorphism group acting regularly on the vertex set $V$. $\mathrm{Cay}(G,H)$ is connected if and only if $G$ is generated by $H$.
	
The spectrum of a Cayley digraph over an abelian group can be very conveniently expressed in terms of the representation theory of the underlying group. Let $\chi$ be a character of a finite abelian group $G$. Then $\big(\chi(x)\big)_{x\in G}$ is an eigenvector of the adjacency matrix $A$ of the Cayley digraph $\mathrm{Cay}(G,H)$ with eigenvalue
$$\chi(H):=\sum_{y\in H}\chi(y).$$
By the orthogonality relations for characters, the distinct characters $\chi_g$, $g\in G$, of $G$ give $|G|$  mutually orthogonal eigenvectors of $A$, and so $A$ can be diagonalized by a unitary similarity transformation. Specifically, take the $|G|\times |G|$ unitary matrix
	$$P:=\frac{1}{\sqrt {|G|}}\big(c_{h,g}\big)_{h,g\in G},\ \ \ {\rm where}\ c_{h,g}=\chi_g(h).$$
Then
$$P^* A P=\diag\big(\chi_g(H)\big)_{g\in G}$$ is a diagonal matrix, where $P^*$ means the conjugate transpose of $P$. Note that since $\chi_g(h)=\chi_h(g)$ for any $h,g\in G$, $P$ is a symmetric matrix.

\subsection{Bi-Cayley graphs}\label{sec:Bi-Cay}
	
A graph is said to be a \emph{bi-Cayley graph} (or {\em semi-Cayley graph}) over a group $G$ if it admits $G$ as a semiregular automorphism group with two orbits of equal size. Every bi-Cayley graph admits the following concrete realization (see \cite[Lemma 2.1]{rj}). Let $R,L$ and $T$ be subsets of a group $G$ such that $R=R^{-1}$, $L=L^{-1}$ and $R\cup L$ does not contain the identity element of $G$.
	Define the graph $\mathrm{BiCay}(G; R, L, T)$ to have vertex set the union of the right part $G_0=\{g_0\mid g\in G\}$ and the left part $G_1=\{g_1\mid g\in G\}$, and edge set the union of the right edges $\{\{h_0,g_0\}\mid gh^{-1}\in R\}$, the left edges $\{\{h_1,g_1\}\mid gh^{-1}\in L\}$ and the spokes $\{\{h_0,g_1\}\mid gh^{-1}\in T\}$. For convenience, for $g\in G$, when we say $g\in G_i$ with $i=0,1$, it means $g_i\in G_i$.
	
\begin{pro}{\rm \label{bicayleyadjacencymatrix} \cite[Lemma 3.1] {The-spectrum-of-semi-Cayley-graphs-over-abelian-groups}}
Let $\Gamma=\mathrm{BiCay}(G;R, L, T)$ be a bi-Cayley graph over an abelian group $G$. Let $A,C,B$ and $D$ be the adjacency matrices of $\mathrm{Cay}(G,R)$, $\mathrm{Cay}(G,L)$, $\mathrm{Cay}(G,T)$ and $\Gamma$, respectively. Then
$$D={
\left ( \begin{array}{cc}
				A & B\\
				B^T & C
			\end{array}
\right )},
$$
where $B^T$ means the transpose of $B$.
\end{pro}
	
Note that due to $R=R^{-1}$ and $L=L^{-1}$, the matrices $A$ and $C$ in Proposition \ref{bicayleyadjacencymatrix} are real and symmetric, and so is $D$.

\section{Transfer matrices of abelian bi-Cayley graphs}
	
Since the adjacency matrix $D$ of a bi-Cayley graph $\Gamma$ is real and symmetric, there exists a unitary matrix $U$ such that $D$ is unitarily similar to a diagonal matrix $\Lambda$, that is, $D=U \Lambda U^*$. Thus the transfer matrix of $\Gamma$ at the time $t$ is
	$$\exp(-\mathrm{i}tD)=\sum\limits_{s=0}^{+\infty}\frac{(-\mathrm{i}tD)^s}{s!}=U \exp(-\mathrm{i}t\Lambda) U^*.$$
We shall characterize the matrices $U$ and $\Lambda$ in this section to give an explicit expression for every entry in $\exp(-\mathrm{i}tD)$.
	
We always adopt the following notation. Let $\Gamma=\mathrm{BiCay}(G;R, L, T)$ be a bi-Cayley graph over a finite abelian group $G$. Let $A,C,B$ and $D$ be the adjacency matrices of $\mathrm{Cay}(G,R)$, $\mathrm{Cay}(G,L)$, $\mathrm{Cay}(G,T)$ and $\Gamma$, respectively. Let $\{\chi_g\mid g\in G\}$ be the set of $|G|$ distinct characters of $G$, and set
	$$P:=\frac{1}{\sqrt {|G|}}\big(c_{h,g}\big)_{h,g\in G},\ \ \ {\rm where}\ c_{h,g}=\chi_g(h).$$
We use the convention that $\chi_g(\emptyset)=0$ for any $g\in G$. By Section \ref{sec:Cayley},
	$$\Lambda_A:=\diag\big(\chi_g(R)\big)_{g\in G}=P^* A P,$$
	$$\Lambda_C:=\diag\big(\chi_g(L)\big)_{g\in G}=P^* C P,$$
	$$\Lambda_B:=\diag\big(\chi_g(T)\big)_{g\in G}=P^* B P.$$
Since $P^* B^* P=\Lambda_B^*$ and $B$ is a $(0,1)$-matrix, $P^* B^T P=\overline{\Lambda_B}$. Take the unitary matrix \begin{equation}\label{eqn:U}
		\hat{U}:={
			\left ( \begin{array}{cc}
				P & 0\\
				0 & P
			\end{array}
			\right )}.
	\end{equation}
	Then
	\begin{equation}\label{eqn:U^* DU}
		\hat{U}^* D\hat{U}={
			\left ( \begin{array}{cc}
				P^* & 0\\
				0 & P^*
			\end{array}
			\right )
			\left ( \begin{array}{cc}
				A & B\\
				B^T & C
			\end{array}
			\right )
			\left ( \begin{array}{cc}
				P & 0\\
				0 & P
			\end{array}
			\right )=
			\left ( \begin{array}{cc}
				\Lambda_A & \Lambda_B\\
				\overline{\Lambda_B} & \Lambda_C
			\end{array}
			\right ):=\hat{D}}.
	\end{equation}
	
We shall diagonalize the matrix $\hat{D}$ by calculating its eigenvalues and eigenvectors in Section \ref{The eigenvalues and eigenvectors of bi-Cayley graphs}. This procedure also provides us with eigenvalues and eigenvectors of the adjacency matrix $D$ of $\Gamma$. Then a practical expression for each entry of the transfer matrix of $\Gamma$ is given in Section \ref{The transfer matrix of bi-Cayley graphs}.

\subsection{Eigenvalues and eigenvectors of abelian bi-Cayley graphs}\label{The eigenvalues and eigenvectors of bi-Cayley graphs}
	
We always assume that $G=\{g_1,g_2,\ldots,g_n\}$ is an abelian group of order $n$. $\hat{D}$ is a $2n\times 2n$ Hermitian matrix. Denote by $\hat{G}=\{\chi_{g_1},\chi_{g_2},\ldots\chi_{g_n}\}$ the character group of $G$. For convenience, $\chi_{k}$ is often used instead of $\chi_{g_k}$.
	
\begin{lem}\label{lem:eigenvalues and eigenvectors of Dhat}
	\begin{itemize}
		\item[$1)$] The $2n$ eigenvalues of $\hat{D}$ are
			\begin{equation}\label{eqn:eigenvalues}
				\lambda_{2k-j}=\frac{\chi_k(R)+\chi_k(L)+(-1)^j\sqrt{\big(\chi_k(R)-\chi_k(L)\big)^2+4|\chi_k(T)|^2}}{2},
			\end{equation}
			where $k=1,2,\ldots,n$ and $j=0,1$.
		\item[$2)$] An eigenvector of $\hat{D}$ associated with the eigenvalue $\lambda_{2k-j}$ is
			\begin{equation}\label{eqn:eigenvectors}
				u_{2k-j}=\frac{1}{\sqrt{|\xi_{k,2+(-1)^j}|^2+|\xi_{k,3+(-1)^j}|^2}}(0,\cdots,0,\underset{k\text{-th}}{\xi_{k,2+(-1)^j}},0,\cdots,0,\underset{(n+k)\text{-th}}{\xi_{k,3+(-1)^j}},0\cdots,0)^T,
			\end{equation}
			where
			\begin{equation}\label{eqn:xi1}
				(\xi_{k,1},\xi_{k,2})=\left\{
				\begin{array}{ll}
					(0,1), & {\rm{if}}\ \chi_k(T)=0,\ \chi_k(R)> \chi_k(L); \\
					(1,0), & {\rm{if}}\ \chi_k(T)=0,\ \chi_k(R)\leq \chi_k(L); \\
					(\frac{\chi_k(T)}{\lambda_{2k-1}-\chi_k(R)},1), & {\rm{if}}\ \chi_k(T)\neq 0,
				\end{array}
				\right.
			\end{equation}
			and
			\begin{equation}\label{eqn:xi2}
				(\xi_{k,3},\xi_{k,4})=\left\{
				\begin{array}{ll}
					(1,0), & {\rm{if}}\ \chi_k(T)=0,\ \chi_k(R)> \chi_k(L); \\
					(0,1), & {\rm{if}}\ \chi_k(T)=0,\ \chi_k(R)\leq \chi_k(L); \\
					(1,\frac{\overline{\chi_k(T)}}{\lambda_{2k}-\chi_k(L)}), & {\rm{if}}\ \chi_k(T)\neq 0.
				\end{array}
				\right.
			\end{equation}
			The $2n$ eigenvectors $u_{2k-j}$, $k=1,2,\ldots,n$ and $j=0,1$, are mutually orthogonal.
		\end{itemize}
	\end{lem}
	
\begin{proof} 1) Let $\lambda$ be an eigenvalue of $\hat{D}$. Then
	\begin{align*}
		\begin{split}
			0=|\lambda I-\hat{D}|  = & \left |
			\setlength{\arraycolsep}{0.5pt}
			\begin{array}{cccccccc}
				\lambda-\chi_1(R)& 0&\cdots&0&-\chi_1(T)&0&\cdots&0\\
				0&\lambda-\chi_2(R)&\cdots&0&0&-\chi_2(T)&\cdots&0\\
				\vdots&\vdots&\ddots&\vdots&\vdots&\vdots&\ddots&\vdots\\
				0&0&\cdots&\lambda-\chi_n(R)&0&0&\cdots&-\chi_n(T)\\
				-\overline{\chi_1(T)}&0&\cdots&0&\lambda-\chi_1(L)& 0&\cdots&0\\
				0&-\overline{\chi_2(T)}&\cdots&0&0&\lambda-\chi_2(L)&\cdots&0\\
				\vdots&\vdots&\ddots&\vdots&\vdots&\vdots&\ddots&\vdots\\
				0&0&\cdots&-\overline{\chi_n(T)}&0&0&\cdots&\lambda-\chi_n(L)
			\end{array}\right|.
		\end{split}
	\end{align*}
	It follows that
	$$
	\begin{vmatrix}
		\lambda-\chi_k(R)&-\chi_k(T)\\
		-\overline{\chi_k(T)}&\lambda-\chi_k(L)
	\end{vmatrix}=0
	$$
	for $1\leq k\leq n$, which yields the eigenvalues of $\hat{D}$ shown in \eqref{eqn:eigenvalues}.
	
	2) For $1\leq k\leq n$, write
	$$\hat{D}_k=\begin{pmatrix}
		\chi_k(R)&\chi_k(T)\\
		\overline{\chi_k(T)}&\chi_k(L)\\
	\end{pmatrix}.$$
	Then $\lambda_{2k-1}$ and $\lambda_{2k}$ are eigenvalues of $\hat{D}_k$. If $(\xi_{k,2+(-1)^j},\xi_{k,3+(-1)^j})^T$ is an eigenvector of $\hat{D}_k$ associated with the eigenvalue $\lambda_{2k-j}$, $j=0,1$, then $u_{2k-j}$ given in \eqref{eqn:eigenvectors} is an eigenvector of $\hat{D}$ associated with the eigenvalue $\lambda_{2k-j}$. So it suffices to calculate the eigenvector $(\xi_{k,2+(-1)^j},\xi_{k,3+(-1)^j})^T$ of $\hat{D}_k$ associated with the eigenvalue $\lambda_{2k-j}$.
	
If $\chi_k(T)\neq 0$, then using \eqref{eqn:eigenvalues} one can check that $\lambda_{2k-j}\not\in\{\chi_k(R),\chi_k(L)\}$ for any $j\in\{0,1\}$, and so $\lambda_{2k-1}-\chi_k(R)$ and $\lambda_{2k}-\chi_k(L)$ can be used as denominators. This makes us to choose $(\xi_{k,1},\xi_{k,2})=(\frac{\chi_k(T)}{\lambda_{2k-1}-\chi_k(R)},1)$ and $(\xi_{k,3},\xi_{k,4})=
(1,\frac{\overline{\chi_k(T)}}{\lambda_{2k}-\chi_k(L)})$. Note that $\lambda_{2k-1}$ and $\lambda_{2k}$ are two distinct eigenvalues of $\hat{D}_k$ (otherwise $\chi_k(T)=0$), and $\hat{D}_k$ is a Hermitian matrix, so $(\xi_{k,1},\xi_{k,2})^T$ and $(\xi_{k,3},\xi_{k,4})^T$ are orthogonal.
	
If $\chi_k(T)=0$, when $\chi_k(R)> \chi_k(L)$, $\lambda_{2k-1}=\chi_k(L)$ and $\lambda_{2k}=\chi_k(R)$, and so $(\xi_{k,1},\xi_{k,2})=(0,1)$ and $(\xi_{k,3},\xi_{k,4})=(1,0)$; when $\chi_k(R)\leq \chi_k(L)$, $\lambda_{2k-1}=\chi_k(R)$ and $\lambda_{2k}=\chi_k(L)$, and so $(\xi_{k,1},\xi_{k,2})=(1,0)$ and $(\xi_{k,3},\xi_{k,4})=(0,1)$.
\end{proof}	
	
\begin{thm}\label{thm:eigenvalues and eigenvectors of D}
Let $\Gamma$ be a bi-Cayley graph over an abelian group of order $n$ with the adjacency matrix $D$. Then $D$ has eigenvalues $\lambda_{2k-j}$ given in \eqref{eqn:eigenvalues} where $k=1,2,\ldots,n$ and $j=0,1$, and an eigenvector of $D$ associated with $\lambda_{2k-j}$ is $\hat{U}u_{2k-j}$ where $\hat{U}$ and $u_{2k-j}$ are given in \eqref{eqn:U} and \eqref{eqn:eigenvectors}, respectively. These eigenvectors $\hat{U}u_{2k-j}$ for $k=1,2,\ldots,n$ and $j=0,1$ are mutually orthogonal.
\end{thm}
	
\begin{proof} By \eqref{eqn:U^* DU}, $D$ and $\hat{D}$ have the same eigenvalues, so by Lemma \ref{lem:eigenvalues and eigenvectors of Dhat}, $D$ has eigenvalues $\lambda_{2k-j}$ for $k=1,2,\ldots,n$ and $j=0,1$. By Lemma \ref{lem:eigenvalues and eigenvectors of Dhat}, $\hat{D} u_{2k-j}=\lambda_{2k-j} u_{2k-j}$. Since $\hat{U}^* D\hat{U} =\hat{D}$, we have $\hat{U}^* D\hat{U} u_{2k-j}=\lambda_{2k-j} u_{2k-j}$, which yields $D\hat{U} u_{2k-j}=\lambda_{2k-j}\hat{U} u_{2k-j}$. It follows that $\hat{U}u_{2k-j}$ is an eigenvector of $D$ associated with $\lambda_{2k-j}$. The orthogonality of $\hat{U}u_{2k-j}$ for $k=1,2,\ldots,n$ and $j=0,1$ comes from the orthogonality of $u_{2k-j}$ in Lemma \ref{lem:eigenvalues and eigenvectors of Dhat}. \end{proof}
	
We remark that eigenvalues and eigenvectors of a bi-Cayley graph over a finite abelian group were also obtained in \cite[Theorem 3.2]{The-spectrum-of-semi-Cayley-graphs-over-abelian-groups} and \cite[Lemma 11]{AT}, respectively, but we here provide a short proof to count eigenvalues and give simple explicit expressions for eigenvectors.

\subsection{Characterization of transfer matrices}\label{The transfer matrix of bi-Cayley graphs}
	
\begin{lem}\label{lem:transfer matrix}
Let $\Gamma$ be a bi-Cayley graph over an abelian group $G$ of order $n$ with the adjacency matrix $D$. Theorem $\ref{thm:eigenvalues and eigenvectors of D}$ shows that $D$ has eigenvalues $\lambda_{2k-j}$ given in \eqref{eqn:eigenvalues} where $k=1,2,\ldots,n$ and $j=0,1$. Let $\hat{G}=\{\chi_{1},\chi_{2},\ldots\chi_{n}\}$ be the character group of $G$. For $k=1,2,\ldots,n$ and $j=0,1$, write
		\begin{equation}\label{eqn:hatxi}
			\hat{\xi}_{k,2+(-1)^j}:=\frac{\xi_{k,2+(-1)^j}}{\sqrt{|\xi_{k,2+(-1)^j}|^2+|\xi_{k,3+(-1)^j}|^2}},
		\end{equation}
		and
		\begin{equation}\label{eqn:hatxi1}
			\hat{\xi}_{k,3+(-1)^j}:=\frac{\xi_{k,3+(-1)^j}}{\sqrt{|\xi_{k,2+(-1)^j}|^2+|\xi_{k,3+(-1)^j}|^2}},
		\end{equation}
		where $\xi_{k,2+(-1)^j}$ and $\xi_{k,3+(-1)^j}$ are given in \eqref{eqn:xi1} and \eqref{eqn:xi2}. Then the transfer matrix of $\Gamma$ at the time $t$ is
		$$
			H(t)=\exp(-\mathrm{i}tD)=U \exp{(-\mathrm{i}t\Lambda)}U^*,
		$$
		where $\Lambda=\diag(\lambda_1,\lambda_2,\ldots,\lambda_{2n})$ is a diagonal matrix, and
$$
			U=\frac{1}{\sqrt{n}}\left(
			\begin{array}{cccccccc}
				\hat{\xi}_{11}\chi_1(g_1)&\hat{\xi}_{13}\chi_1(g_1)&\hat{\xi}_{21}\chi_2(g_1)&\hat{\xi}_{23}\chi_2(g_1)&\cdots&\hat{\xi}_{n1}\chi_n(g_1)&\hat{\xi}_{n3}\chi_n(g_1)&\\
				\hat{\xi}_{11}\chi_1(g_2)&\hat{\xi}_{13}\chi_1(g_2)&\hat{\xi}_{21}\chi_2(g_2)&\hat{\xi}_{23}\chi_2(g_2)&\cdots&\hat{\xi}_{n1}\chi_n(g_2)&\hat{\xi}_{n3}\chi_n(g_2)&\\
				\vdots&\vdots&\vdots&\vdots&\ddots&\vdots&\vdots&\\
				\hat{\xi}_{11}\chi_1(g_n)&\hat{\xi}_{13}\chi_1(g_n)&\hat{\xi}_{21}\chi_2(g_n)&\hat{\xi}_{23}\chi_2(g_n)&\cdots&\hat{\xi}_{n1}\chi_n(g_n)&\hat{\xi}_{n3}\chi_n(g_n)&\\
				\hat{\xi}_{12}\chi_1(g_1)&\hat{\xi}_{14}\chi_1(g_1)&\hat{\xi}_{22}\chi_2(g_1)&\hat{\xi}_{24}\chi_2(g_1)&\cdots&\hat{\xi}_{n2}\chi_n(g_1)&\hat{\xi}_{n4}\chi_n(g_1)&\\
				\hat{\xi}_{12}\chi_1(g_2)&\hat{\xi}_{14}\chi_1(g_2)&\hat{\xi}_{22}\chi_2(g_2)&\hat{\xi}_{24}\chi_2(g_2)&\cdots&\hat{\xi}_{n2}\chi_n(g_2)&\hat{\xi}_{n4}\chi_n(g_2)&\\
				\vdots&\vdots&\vdots&\vdots&\ddots&\vdots&\vdots&\\
				\hat{\xi}_{12}\chi_1(g_n)&\hat{\xi}_{14}\chi_1(g_n)&\hat{\xi}_{22}\chi_2(g_n)&\hat{\xi}_{24}\chi_2(g_n)&\cdots&\hat{\xi}_{n2}\chi_n(g_n)&\hat{\xi}_{n4}\chi_n(g_n)&\\
			\end{array}\right).
$$
	\end{lem}
	
\begin{proof} By Theorem $\ref{thm:eigenvalues and eigenvectors of D}$, $D$ has eigenvalues $\lambda_{2k-j}$ ($k=1,2,\ldots,n$ and $j=0,1$), and an eigenvector of $D$ associated with $\lambda_{2k-j}$ is $\hat{U}u_{2k-j}$, where
	\begin{center}
		$\hat{U}=
		\left ( \begin{array}{cc}
			P & 0\\
			0 & P
		\end{array}
		\right )
		=\frac{1}{\sqrt{n}}\left(\begin{array}{cccccccc}
			\chi_1(g_1)&\chi_2(g_1)&\cdots&\chi_n(g_1)&0&0&\cdots&0\\
			\chi_1(g_2)&\chi_2(g_2)&\cdots&\chi_n(g_2)&0&0&\cdots&0\\
			\vdots&\vdots&\ddots&\vdots&\vdots&\vdots&\ddots&\vdots\\
			\chi_1(g_n)&\chi_2(g_n)&\cdots&\chi_n(g_n)&0&0&\cdots&0\\
			0&0&\cdots&0&\chi_1(g_1)&\chi_2(g_1)&\cdots&\chi_n(g_1)\\
			0&0&\cdots&0&\chi_1(g_2)&\chi_2(g_2)&\cdots&\chi_n(g_2)\\
			\vdots&\vdots&\ddots&\vdots&\vdots&\vdots&\ddots&\vdots\\
			0&0&\cdots&0&\chi_1(g_n)&\chi_2(g_n)&\cdots&\chi_n(g_n)\\
		\end{array}
		\right)$,
	\end{center}
	and
	$$
	u_{2k-j}=(0,\cdots,0,\underset{k\text{-th}}{\hat{\xi}_{k,2+(-1)^j}},0,\cdots,0,\underset{(n+k)\text{-th}}{\hat{\xi}_{k,3+(-1)^j}},0\cdots,0)^T.
	$$
	Let $\Lambda=\diag(\lambda_1,\lambda_2,\ldots,\lambda_{2n})$ and $U=(\hat{U}u_1,\hat{U}u_2,\ldots,\hat{U}u_{2n})$. Then $U^* D U=\Lambda$. Since $U$ is a unitary matrix, we have $D=U \Lambda U^*$. Thus $H(t)=\exp(-\mathrm{i}tD)=U \exp{(-\mathrm{i}t\Lambda)}U^*$. \end{proof}
	
By Lemma \ref{lem:transfer matrix}, we can give an explicit expression for every entry in the transfer matrix $H(t)$ as follows.

\begin{core}\label{core:H-entry}
We adopt the notation of Lemma $\ref{lem:transfer matrix}$. Let $H(t)$ be the transfer matrix of $\Gamma$ at the time $t$. The rows and columns of $H(t)$ are labeled with the values in $G_0$ and $G_1$ which are defined in Section $\ref{sec:Bi-Cay}$. Let $H_{g_p,g_q}(t)$ be the element in the $g_p$-th row and the $g_q$-th column of the matrix $H(t)$. Then
		$$
		nH_{g_p,g_q}(t)=\left\{
		\begin{array}{ll} \sum\limits_{k=1}^{n}\chi_{k}(g_pg_q^{-1})(|\hat{\xi}_{k,1}|^2\exp(-\mathrm{i}t\lambda_{2k-1})+|\hat{\xi}_{k,3}|^2\exp(-\mathrm{i}t\lambda_{2k})), & {\rm{if}}\ (g_p,g_q)\in G_0\times G_0; \\					\sum\limits_{k=1}^{n}\chi_{k}(g_pg_q^{-1})(|\hat{\xi}_{k,2}|^2\exp(-\mathrm{i}t\lambda_{2k-1})+|\hat{\xi}_{k,4}|^2\exp(-\mathrm{i}t\lambda_{2k})),
			& {\rm{if}}\ (g_p,g_q)\in G_1\times G_1; \\              					\sum\limits_{k=1}^{n}\chi_{k}(g_pg_q^{-1})(\hat{\xi}_{k,1}\overline{\hat{\xi}_{k,2}}\exp(-\mathrm{i}t\lambda_{2k-1})+\hat{\xi}_{k,3}\overline{\hat{\xi}_{k,4}}\exp(-\mathrm{i}t\lambda_{2k})),
			& {\rm{if}}\ (g_p,g_q)\in G_0\times G_1;\\					\sum\limits_{k=1}^{n}\chi_{k}(g_pg_q^{-1})(\hat{\xi}_{k,2}\overline{\hat{\xi}_{k,1}}\exp(-\mathrm{i}t\lambda_{2k-1})+\hat{\xi}_{k,4}\overline{\hat{\xi}_{k,3}}\exp(-\mathrm{i}t\lambda_{2k})),
			& {\rm{if}}\ (g_p,g_q)\in G_1\times G_0.	
		\end{array}
		\right.
		$$
	\end{core}
	
\begin{proof} Note that the \emph{$g_p$-th} row of the matrix $U$ given in Lemma \ref{lem:transfer matrix} is:
	\begin{equation*}
		\begin{array}{lr}
			g_p\in G_0:
			~\frac{1}{\sqrt n}(\hat{\xi}_{11}\chi_1(g_p),\hat{\xi}_{13}\chi_1(g_p),\hat{\xi}_{21}\chi_2(g_p),\hat{\xi}_{23}\chi_2(g_p),\cdots,\hat{\xi}_{n1}\chi_n(g_p),\hat{\xi}_{n3}\chi_n(g_p))\\[3mm]
			g_p\in G_1:
			~\frac{1}{\sqrt n}(\hat{\xi}_{12}\chi_1(g_p),\hat{\xi}_{14}\chi_1(g_p),\hat{\xi}_{22}\chi_2(g_p),\hat{\xi}_{24}\chi_2(g_p),\cdots,\hat{\xi}_{n2}\chi_n(g_p),\hat{\xi}_{n4}\chi_n(g_p)),\\
		\end{array}
	\end{equation*}
	and $\exp{(-\mathrm{i}t\Lambda)}=\diag(\exp{(-\mathrm{i}t\lambda_1}),\exp{(-\mathrm{i}t\lambda_2}),\ldots,\exp{(-\mathrm{i}t\lambda_{2n}}))$. Then one can apply Lemma \ref{lem:transfer matrix} to complete the proof. \end{proof}

\section{Perfect state transfer on abelian bi-Cayley graphs}
	
This section is devoted to providing necessary and sufficient conditions for a bi-Caylay graph over an abelian group to have PST.

Recall that at the beginning of Section \ref{The eigenvalues and eigenvectors of bi-Cayley graphs}, we make the convention that $G=\{g_1,g_2,\ldots,g_n\}$ is an abelian group with the identity element $g_1$, and $\hat{G}=\{\chi_{g_1},\chi_{g_2},\ldots\chi_{g_n}\}$ is the character group of $G$. For $1\leq k\leq n$, $\chi_{g_k}$ is simply written as $\chi_{k}$ and $\chi_{k^{-1}}$ is often used instead of $\chi_{g_k^{-1}}$. So $\chi_{1}$ is the principal character that maps every element of $G$ to $1$. Also we use the convention that $\chi_g(\emptyset)=0$ for any $g\in G$.
	
\begin{lem}\label{PSTiff}
Let $\Gamma=\mathrm{BiCay}(G;R, L, T)$ be a bi-Cayley graph over an abelian group $G$ of order $n$. Then $\Gamma$ has PST between vertices $g_p$ and $g_q$ at the time $t$ if and only if the following conditions hold for each $1\leq k\leq n$, where $\lambda_1,\lambda_2,\ldots,\lambda_{2n}$ are given in Lemma $\ref{lem:eigenvalues and eigenvectors of Dhat}$, and
$$X=\left\{
\begin{array}{ll}
	R, & {\rm{if}}\ (g_p,g_q)\in (G_0\times G_0); \\	[3mm]
	L, & {\rm{if}}\ (g_p,g_q)\in (G_1\times G_1).
\end{array}
\right.
$$

	\begin{itemize}
		\item [$1)$] $$\chi_{k}({g_p}{g_q}^{-1})=\left\{
		\begin{array}{ll}
			\pm 1, & {\rm{if}}\ (g_p,g_q)\in (G_0\times G_0)\cup (G_1\times G_1); \\	[3mm]
			
			\frac{|\chi_k(T)|}{\chi_k(T)}, & {\rm{if}}\ (g_p,g_q)\in G_0\times G_1\  {\rm{and}} \ \exp(\rm{i}t(\lambda_{2}-\lambda_{2{k}}))=1 ;\\[3mm]
			
			-\frac{|\chi_k(T)|}{\chi_k(T)}, & {\rm{if}}\ (g_p,g_q)\in G_0\times G_1\  {\rm{and}} \ \exp(\rm{i}t(\lambda_{2}-\lambda_{2{k}}))=-1 ;\\[3mm]
			
			\frac{|\chi_k(T)|}{\overline{\chi_k(T)}}, & {\rm{if}}\ (g_p,g_q)\in G_1\times G_0\  {\rm{and}} \ \exp(\rm{i}t(\lambda_{2}-\lambda_{2{k}}))=1 ;\\[3mm]
			
			-\frac{|\chi_k(T)|}{\overline{\chi_k(T)}}, & {\rm{if}}\ (g_p,g_q)\in  G_1\times G_0\  {\rm{and}} \ \exp(\rm{i}t(\lambda_{2}-\lambda_{2{k}}))=-1.
		\end{array}
		\right.
		$$	
		\item [$2)$] If $(g_p,g_q)\in (G_0\times G_0)\cup (G_1\times G_1)$, then
		
		$$\left\{
		\begin{array}{ll}
			t(\lambda_{2k}-\lambda_{2k-1})\in\{2z\pi\mid z\in \ZZ\}\ \ {\rm{and}} &  \\
			\ \ \ t(\lambda_{2}-\lambda_{2k-1})\in \{(2z-\frac{\chi_{k}({g_p}{g_q}^{-1})-1}{2})\pi\mid z\in \ZZ\}, & {\rm{if}}\ \chi_k(T)\neq 0; \\	[3mm]
			t(|X|-\chi_k(X))\in\{(2z-\frac{\chi_{k}({g_p}{g_q}^{-1})-1}{2})\pi\mid z\in \ZZ\}, & {\rm{if}}\ T=\emptyset;\\	[3mm] t(\lambda_{2}-\chi_k(X))\in\{(2z-\frac{\chi_{k}({g_p}{g_q}^{-1})-1}{2})\pi\mid z\in \ZZ\}, & {\rm{if}}\ \chi_k(T)= 0, T\neq\emptyset.
		\end{array}
		\right.
		$$	
		
\item [$3)$] If $(g_p,g_q)\in (G_0\times G_1)\cup (G_1\times G_0)$, then $R=L$, $\chi_k(T)\neq 0$, $t(\lambda_{2k}-\lambda_{2k-1})\in \{(2z+1)\pi\mid z\in \ZZ\}$ and  $t(\lambda_{2}-\lambda_{2k})\in \{z\pi\mid z\in \ZZ\}$.
	\end{itemize}
\end{lem}

\begin{proof} {\bf Case 1}: First we consider the case $(g_p,g_q)\in G_0\times G_0$. By Corollary \ref{core:H-entry},
\begin{equation}\label{eqn:|H00|=n}
	|H_{g_p,g_q}(t)|=1 \Longleftrightarrow\big|\sum\limits_{k=1}^{n}\chi_{k}({g_p}{g_q}^{-1})(|\hat{\xi}_{k,1}|^2\exp(-{\mathrm i}t\lambda_{2k-1})+|\hat{\xi}_{k,3}|^2\exp(-{\mathrm i}t\lambda_{2k}))\big|=n.
\end{equation}
Since $|\chi_{k}({g_p}{g_q}^{-1})|=1$ and $\lambda_{2k-1}$, $\lambda_{2k}$ are real numbers,
$$\big|\chi_{k}({g_p}{g_q}^{-1})(|\hat{\xi}_{k,1}|^2\exp(-{\mathrm i}t\lambda_{2k-1})+|\hat{\xi}_{k,3}|^2\exp(-{\mathrm i}t\lambda_{2k}))\big|\leq |\hat{\xi}_{k,1}|^2+|\hat{\xi}_{k,3}|^2.$$ By \eqref{eqn:xi1} and \eqref{eqn:xi2} together with the use of \eqref{eqn:eigenvalues}, we have $\xi_{k,2}=\xi_{k,3}$ and $|\xi_{k,1}|^2+|\xi_{k,2}|^2=|\xi_{k,3}|^2+|\xi_{k,4}|^2$. Thus \begin{equation}\label{eqn:sumsqr}
	|\hat{\xi}_{k,1}|^2+|\hat{\xi}_{k,3}|^2=\frac{|\xi_{k,1}|^2}{|\xi_{k,1}|^2+|\xi_{k,2}|^2}+\frac{|\xi_{k,3}|^2}{|\xi_{k,3}|^2+|\xi_{k,4}|^2}=1,
\end{equation}
which yields
$ \big|\chi_{k}({g_p}{g_q}^{-1})(|\hat{\xi}_{k,1}|^2\exp(-{\mathrm i}t\lambda_{2k-1})+|\hat{\xi}_{k,3}|^2\exp(-{\mathrm i}t\lambda_{2k}))\big|\leq 1.
$
The summation in \eqref{eqn:|H00|=n} has $n$ terms and the modulus of each term is no more than $1$, so $|H_{g_p,g_q}(t)| = 1$ if and only if for any $k\in\{1,2,\ldots,n\}$,
\begin{equation}\label{eqn:H00l=1}
	\big|\chi_{k}({g_p}{g_q}^{-1})(|\hat{\xi}_{k,1}|^2\exp(-{\mathrm i}t\lambda_{2k-1})+|\hat{\xi}_{k,3}|^2\exp(-{\mathrm i}t\lambda_{2k}))\big|=1,
\end{equation}
and for any $k,k'\in\{1,2,\ldots,n\}$,
\begin{align}\label{eqn:H00i=l}
	& \chi_{k}({g_p}{g_q}^{-1})(|\hat{\xi}_{k,1}|^2\exp(-{\mathrm i}t\lambda_{2k-1})+|\hat{\xi}_{k,3}|^2\exp(-{\mathrm i}t\lambda_{2k}))\notag\\
	= & \chi_{k'}({g_p}{g_q}^{-1})(|\hat{\xi}_{{k'},1}|^2\exp(-{\mathrm i}t\lambda_{2{k'}-1})+|\hat{\xi}_{{k'},3}|^2\exp(-{\mathrm i}t\lambda_{2{k'}})).
\end{align}

Clearly \eqref{eqn:H00l=1} holds if and only if
$\big||\hat{\xi}_{k,1}|^2\exp(-{\mathrm i}t\lambda_{2k-1})+
|\hat{\xi}_{k,3}|^2\exp(-{\mathrm i}t\lambda_{2k})\big|=1$.
By \eqref{eqn:sumsqr}, $|\hat{\xi}_{k,1}|^2+|\hat{\xi}_{k,3}|^2=1$, which yields $|\hat{\xi}_{k,1}|^2 |\exp(-{\mathrm i}t\lambda_{2k-1})|+|\hat{\xi}_{k,3}|^2 |\exp(-{\mathrm i}t \lambda_{2k})|=1$. It follows that \eqref{eqn:H00l=1} holds if and only if
\begin{equation}\label{eqn:7-25-2}
	\big||\hat{\xi}_{k,1}|^2\exp(-{\mathrm i}t\lambda_{2k-1})+
	|\hat{\xi}_{k,3}|^2\exp(-{\mathrm i}t\lambda_{2k})\big|=|\hat{\xi}_{k,1}|^2 |\exp(-{\mathrm i}t\lambda_{2k-1})|+|\hat{\xi}_{k,3}|^2 |\exp(-{\mathrm i}t \lambda_{2k})|=1.
\end{equation}
If $\chi_k(T)=0$, by \eqref{eqn:xi1} and \eqref{eqn:xi2}, \eqref{eqn:7-25-2} always holds. If $\chi_k(T)\neq 0$, according to the triangle inequality on complex numbers, \eqref{eqn:7-25-2} holds if and only if there exists a real positive number $\alpha_k$ such that $|\hat{\xi}_{k,1}|^2\exp(-{\mathrm i}t\lambda_{2k-1})=\alpha_k |\hat{\xi}_{k,3}|^2\exp(-{\mathrm i}t\lambda_{2k})$. Taking the modulus on both sides, we have
\begin{equation}\label{eqn:alphak1}
	\exp(-{\mathrm i}t\lambda_{2k-1})=\exp(-{\mathrm i}t\lambda_{2k}),
\end{equation}
which implies $t(\lambda_{2k}-\lambda_{2k-1})\in \{2z\pi\mid z\in \ZZ\}$.

Taking $k'=1$ in \eqref{eqn:H00i=l}, we have for any $k\in\{1,2,\ldots,n\}$,
\begin{equation}\label{eqn:7-25-1}
	\chi_{k}({g_p}{g_q}^{-1})(|\hat{\xi}_{k,1}|^2\exp(-{\mathrm i}t\lambda_{2k-1})+|\hat{\xi}_{k,3}|^2\exp(-{\mathrm i}t\lambda_{2k}))=|\hat{\xi}_{1,1}|^2\exp(-{\mathrm i}t\lambda_{1})+|\hat{\xi}_{1,3}|^2\exp(-{\mathrm i}t\lambda_{2}).
\end{equation}
If $\chi_k(T)\neq 0$, then $\chi_1(T)=|T|\neq 0$. By \eqref{eqn:sumsqr}, under the assumption of \eqref{eqn:alphak1}, \eqref{eqn:7-25-1} holds if and only if
\begin{equation}\label{eqn:chikTneq0}
	\chi_{k}({g_p}{g_q}^{-1})=\exp(-{\mathrm i}t(\lambda_{2}-\lambda_{2k-1})).
\end{equation}
If $\chi_k(T)= 0$, by \eqref{eqn:xi1} and \eqref{eqn:xi2}, the left side of \eqref{eqn:7-25-1} is equal to $\chi_{k}({g_p}{g_q}^{-1})\exp(-\mathrm{i}t\chi_k(R))$. For the right side of \eqref{eqn:7-25-1},
if $T=\emptyset$, then $\chi_1(T)=0$, and so by \eqref{eqn:xi1} and \eqref{eqn:xi2}, the right side of \eqref{eqn:7-25-1} is equal to $\exp(-\mathrm{i}t|R|)$; if $T\neq\emptyset$, then $\chi_1(T)\neq 0$, and so by \eqref{eqn:sumsqr}, in the assumption of \eqref{eqn:alphak1}, the right side of \eqref{eqn:7-25-1} can be always written as $\exp(-\mathrm{i}t\lambda_2)$. On the whole, if $\chi_k(T)= 0$, then in the assumption of \eqref{eqn:alphak1}, \eqref{eqn:7-25-1} holds if and only if
\begin{equation}\label{eqn:8-11}
	\chi_k(g_pg_q^{-1})=\left\{
	\begin{array}{ll}
		\exp(-\mathrm{i}t(|R|-\chi_k(R))),~\text{if }T=\emptyset;\\[3mm] \exp(-\mathrm{i}t(\lambda_{2}-\chi_k(R))),~\text{if }T\neq\emptyset.
	\end{array}
	\right.
\end{equation}

Since $\chi_k(R)$ and $\chi_k(L)$ are both real numbers, $\chi_k(R)=\chi_{k^{-1}}(R)$ and $\chi_k(L)=\chi_{k^{-1}}(L)$. Noting that $\chi_{k^{-1}}(T)=\overline{\chi_{k}(T)}$ and using \eqref{eqn:eigenvalues}, we have $\lambda_{2k^{-1}-1}=\lambda_{2k-1}$. Thus if $\chi_k(T)\neq 0$, by \eqref{eqn:chikTneq0},
\begin{equation}\label{eqn:chik-1Tneq0}
	\chi_{k^{-1}}({g_p}{g_q}^{-1})=\exp(-{\mathrm i}t(\lambda_{2}-\lambda_{2k^{-1}-1})=\exp(-{\mathrm i}t(\lambda_{2}-\lambda_{2k-1})).
\end{equation}	
On the other hand,
\begin{align}\label{eqn:chik-1} \chi_{k^{-1}}({g_p}{g_q}^{-1})=\overline{\chi_{k}({g_p}{g_q}^{-1})}\overset{\eqref{eqn:chikTneq0}}{=}\overline{\exp(-{\mathrm i}t(\lambda_{2}-\lambda_{2k-1}))}.
\end{align}
Combining \eqref{eqn:chik-1Tneq0} and \eqref{eqn:chik-1}, we have
$\exp(-{\mathrm i}t(\lambda_{2}-\lambda_{2k-1}))=\overline{\exp(-{\mathrm i}t(\lambda_{2}-\lambda_{2k-1}))}$.
This implies $\exp(-{\mathrm i}t(\lambda_{2}-\lambda_{2k-1}))=\pm1$.
Hence  if $\chi_k(T)\neq 0$, \eqref{eqn:chikTneq0} holds if and only if $\chi_{k}({g_p}{g_q}^{-1})=\pm1$ and
\begin{equation*}\label{eqn:jj''}
 t(\lambda_{2}-\lambda_{2k-1}                                                        )\in \{(2z-\frac{\chi_{k}({g_p}{g_q}^{-1})-1}{2})\pi\mid z\in \ZZ\}.
\end{equation*}

Similarly,  if $\chi_k(T)= 0$, substituting $k$ by $k^{-1}$ in \eqref{eqn:8-11}, we have that \eqref{eqn:8-11} holds if and if only $\chi_{k}({g_p}{g_q}^{-1})=\pm1$, $t(|R|-\chi_k(R))\in\{(2z-\frac{\chi_{k}({g_p}{g_q}^{-1})-1}{2})\pi\mid z\in \ZZ\}$ when $T=\emptyset$, and
$t(\lambda_{2}-\chi_k(R))\in\{(2z-\frac{\chi_{k}({g_p}{g_q}^{-1})-1}{2})\pi\mid z\in \ZZ\}$ when $\chi_k(T)= 0$ and $T\neq\emptyset$.

{\bf Case 2}: Similar arguments to those for $(g_p,g_q)\in G_0\times G_0$ show that the desired results hold for $(g_p,g_q)\in G_1\times G_1$.

{\bf Case 3}: Now we consider the case $(g_p,g_q)\in (G_0\times G_1)\cup (G_1\times G_0)$. Since $\Gamma$ is an undirected graph whose adjacency matrix is symmetric, $H_{g_p,g_q}(t)=H_{g_q,g_p}(t)$. So without loss of generality, we may assume that $(g_p,g_q)\in G_0\times G_1$.

By Corollary \ref{core:H-entry},
\begin{equation}\label{eqn:H01t}
	H_{g_p,g_q}(t)=\frac{1}{n} \sum\limits_{k=1}^{n}\chi_{k}({g_p}{g_q}^{-1})(\hat{\xi}_{k,1}\overline{\hat{\xi}_{k,2}}\exp(-{\mathrm i}t\lambda_{2k-1})+\hat{\xi}_{k,3}\overline{\hat{\xi}_{k,4}}\exp(-{\mathrm i}t\lambda_{2k})).
\end{equation}	
By \eqref{eqn:xi1} and \eqref{eqn:xi2}, if there exists some $k\in\{1,2,\cdots,n\}$ such that $\chi_k(T)=0$, then either $\xi_{k,1}=\xi_{k,4}=0$ or $\xi_{k,2}=\xi_{k,3}=0$, and so either $\hat{\xi}_{k,1}=\hat{\xi}_{k,4}=0$ or $\hat{\xi}_{k,2}=\hat{\xi}_{k,3}=0$.
Thus \eqref{eqn:H01t} can be written as
\begin{equation}\label{eqn:H01t1}
	H_{g_p,g_q}(t)
	=\frac{1}{n}
	\sum\limits_{k=1,\cdots,n \atop \chi_k(T)\neq 0 }
	\chi_{k}({g_p}{g_q}^{-1})(\hat{\xi}_{k,1}\overline{\hat{\xi}_{k,2}}\exp(-{\mathrm i}t\lambda_{2k-1})+\hat{\xi}_{k,3}\overline{\hat{\xi}_{k,4}}\exp(-{\mathrm i}t\lambda_{2k})).
\end{equation}
When $\chi_k(T)\neq 0$ for some $k\in\{1,2,\cdots,n\}$, by \eqref{eqn:xi1} and \eqref{eqn:xi2}, $\xi_{k,2}=\xi_{k,3}=1$ and ${\xi_{k,4}}=-\overline{{\xi_{k,1}}}$ (this is from the fact that $\lambda_{2k-1}+\lambda_{2k}=\chi_k(R)+\chi_k(L)$ which comes from \eqref{eqn:eigenvalues}). So $\hat{\xi_{k,2}}=\hat{\xi_{k,3}}\in \mathbb{R}$ and $\hat{\xi_{k,4}}=-\overline{\hat{\xi_{k,1}}}$. Then using \eqref{eqn:H01t1} we have
$$
H_{g_p,g_q}(t)
=\frac{1}{n}
\sum\limits_{k=1,\cdots,n \atop \chi_k(T)\neq 0 }\chi_{k}({g_p}{g_q}^{-1})\hat{\xi}_{k,1}\hat{\xi}_{k,2}(\exp(-{\mathrm i}t\lambda_{2k-1})-\exp(-{\mathrm i}t\lambda_{2k})).
$$
It follows that \begin{align}\label{eqn:H01tn}
	|H_{g_p,g_q}(t)|=1&\Longleftrightarrow \Big|
	\sum\limits_{k=1,\cdots,n \atop \chi_k(T)\neq 0 }\chi_{k}({g_p}{g_q}^{-1})\hat{\xi}_{k,1}\hat{\xi}_{k,2}(\exp(-{\mathrm i}t\lambda_{2k-1})-\exp(-{\mathrm i}t\lambda_{2k}))\Big|=n.
\end{align}
Since
$$ \big|\chi_{k}({g_p}{g_q}^{-1})\hat{\xi}_{k,1}\hat{\xi}_{k,2}(\exp(-{\mathrm i}t\lambda_{2k-1})-\exp(-{\mathrm i}t\lambda_{2k}))\big|\leq 2|\hat{\xi}_{k,1}\hat{\xi}_{k,2}|
=\frac{2|\xi_{k,1}\xi_{k,2}|}{|\xi_{k,1}|^2+|\xi_{k,2}|^2}\leq 1,
$$
by \eqref{eqn:H01tn} we have that $|H_{g_p,g_q}(t)| = 1$ if and only if for any $k\in\{1,2,\ldots,n\}$, $\chi_k(T)\neq 0,$
\begin{equation}\label{eqn:H01t=1}
	\big|\chi_{k}({g_p}{g_q}^{-1})\hat{\xi}_{k,1}\hat{\xi}_{k,2}(\exp(-{\mathrm i}t\lambda_{2k-1})-\exp(-{\mathrm i}t\lambda_{2k}))\big|=1,
\end{equation}
and for any $k,k'\in\{1,2,\ldots,n\}$,
\begin{align}\label{eqn:H01i=l}
	& \chi_{k}({g_p}{g_q}^{-1})\hat{\xi}_{k,1}\hat{\xi}_{k,2}(\exp(-{\mathrm i}t\lambda_{2k-1})-\exp(-{\mathrm i}t\lambda_{2k})) \notag\\		= & \chi_{k'}({g_p}{g_q}^{-1})\hat{\xi}_{{k'},1}\hat{\xi}_{{k'},2}(\exp(-{\mathrm i}t\lambda_{2{k'}-1})-\exp(-{\mathrm i}t\lambda_{2{k'}})).
\end{align}	

Clearly \eqref{eqn:H01t=1} holds if and only if $|\hat{\xi}_{k,1}\hat{\xi}_{k,2}|\cdot \big|\exp(-{\mathrm i}t\lambda_{2k-1})-\exp(-{\mathrm i}t\lambda_{2k})\big|=1$. Since
$|\hat{\xi}_{k,1}\hat{\xi}_{k,2}|
=\frac{|\xi_{k,1}\xi_{k,2}|}{|\xi_{k,1}|^2+|\xi_{k,2}|^2}\leq \frac{1}{2}$
and
$\big|\exp(-{\mathrm i}t\lambda_{2k-1})-\exp(-{\mathrm i}t\lambda_{2k})\big|\leq 2,$
the equality \eqref{eqn:H01t=1} holds if and only if
$|\xi_{k,1}|=|\xi_{k,2}|$ and
$
\exp(-{\mathrm i}t\lambda_{2k-1})=-\exp(-{\mathrm i}t\lambda_{2k}).
$
Under the assumption of $\chi_k(T)\neq 0$,
\begin{align}\label{eqn:t1}
	|\xi_{k,1}|=|\xi_{k,2}| \overset{\eqref{eqn:xi1}}{\Longleftrightarrow} |\xi_{k,1}|=|\xi_{k,2}|=1 {\Longleftrightarrow} |\chi_k(T)|=|\lambda_{2k-1}-\chi_k(R)|\overset{\eqref{eqn:eigenvalues}}{\Longleftrightarrow} \chi_k(R)=\chi_k(L),
\end{align}
and
\begin{align}\label{eqn:t2}
	\exp(-{\mathrm i}t\lambda_{2k-1})=-\exp(-{\mathrm i}t\lambda_{2k}) \Longleftrightarrow  t(\lambda_{2k}-\lambda_{2k-1})\in \{(2z+1)\pi\mid z\in \ZZ\}.
\end{align}
Furthermore, \eqref{eqn:t1} holds (i.e., $\chi_k(R)=\chi_k(L)$ for every $k\in\{1,2,\ldots,n\}$) if and only if $R=L$. This is because for each $g\in G$,
$$\sum\limits_{k=1}^{n}\chi_k(g^{-1})\chi_k(R)=
\sum\limits_{k=1}^{n}\sum\limits_{r\in R}\chi_k(g^{-1})\chi_k(r)=
\sum\limits_{r\in R}\sum\limits_{k=1}^{n}\chi_k(g^{-1}r)=
n\sum\limits_{r\in R}\delta_{rg},$$
where $\delta_{rg}=1$ if $r=g$ and $0$ otherwise; similarly, for each $g\in G$, $$\sum\limits_{k=1}^{n}\chi_k(g^{-1})\chi_k(L)=n\sum\limits_{l\in L}\delta_{lg}.$$
Since $\chi_k(R)=\chi_k(L)$ for every $k\in\{1,2,\ldots,n\}$, $\sum\limits_{k=1}^{n}\chi_k(g^{-1})\chi_k(R)=\sum\limits_{k=1}^{n}\chi_k(g^{-1})\chi_k(L)$, so $$n\sum\limits_{r\in R}\delta_{rg}=n\sum\limits_{l\in L}\delta_{lg}$$ for each $g\in G$, which leads to $R=L$.

Taking ${k'}=1$ in \eqref{eqn:H01i=l}, we have that for any $k\in\{1,2,\ldots,n\}$,
\begin{align}\label{eqn:f1}
	\hat{\xi}_{1,1}\hat{\xi}_{1,2}(\exp(-{\mathrm i}t\lambda_{1})-\exp(-{\mathrm i}t\lambda_{2}))
	=\chi_{k}({g_p}{g_q}^{-1})\hat{\xi}_{k,1}\hat{\xi}_{k,2}(\exp(-{\mathrm i}t\lambda_{2k-1})-\exp(-{\mathrm i}t\lambda_{2k})).
\end{align}
In the assumption of \eqref{eqn:t2}, i.e. $\exp(-{\mathrm i}t\lambda_{2k-1})=-\exp(-{\mathrm i}t\lambda_{2k})$, \eqref{eqn:f1} holds if and only if
\begin{align}\label{eqn:f10}
	\chi_k(g_p g_q^{-1})=\frac{\hat{\xi}_{1,1}\hat{\xi}_{1,2}}{\hat{\xi}_{k,1}\hat{\xi}_{k,2}}\cdot \frac{\exp(-{\mathrm i}t\lambda_{2})}{\exp(-{\mathrm i}t\lambda_{2k})}=\frac{\hat{\xi}_{1,1}\hat{\xi}_{1,2}}{\hat{\xi}_{k,1}\hat{\xi}_{k,2}}\cdot \exp(-{\mathrm i}t(\lambda_{2}-\lambda_{2k})).
\end{align}
By \eqref{eqn:t1}, $|\xi_{k,1}|=|\xi_{k,2}|=1$, and by \eqref{eqn:xi1} when $\chi_k(T)\neq 0$, $\xi_{k,2}=1$, so $\hat{\xi}_{k,1}\hat{\xi}_{k,2}=\frac{{\xi}_{k,1}{\xi}_{k,2}}{|{\xi}_{k,1}|^2+|{\xi}_{k,2}|^2}=\frac{\xi_{k,1}}{2}$. which yields $\frac{\hat{\xi}_{1,1}\hat{\xi}_{1,2}}{\hat{\xi}_{k,1}\hat{\xi}_{k,2}}=\frac{{\xi}_{1,1}}{{\xi}_{k,1}}$.
It follows that under the assumption of \eqref{eqn:t1}, \eqref{eqn:f10} holds if and only if
\begin{align}\label{eqn:f20}
	\chi_k(g_p g_q^{-1})& =\frac{\xi_{1,1}}{\xi_{k,1}}\cdot \exp(-{\mathrm i}t(\lambda_{2}-\lambda_{2k})) \overset{\eqref{eqn:xi1}}{=}\frac{\chi_{1}(T)(\lambda_{2k-1}-\chi_{k}(R))}{\chi_k(T)(\lambda_1-\chi_{1}(R))}\cdot \exp(-{\mathrm i}t(\lambda_{2}-\lambda_{2k})).
\end{align}
By \eqref{eqn:t1}, $\chi_k(R)=\chi_k(L)$. Then by \eqref{eqn:eigenvalues} we have $\lambda_{2k-1}-\chi_k(R)=-|\chi_k(T)|$. Thus by \eqref{eqn:f20},
\begin{align}\label{eqn:f30}
	\chi_k(g_p g_q^{-1})&=\frac{\chi_{1}(T)|\chi_k(T)|}{\chi_k(T)|\chi_1(T)|}\cdot \exp(-{\mathrm i}t(\lambda_{2}-\lambda_{2k}))=\frac{|\chi_k(T)|}{\chi_k(T)}\cdot \exp(-{\mathrm i}t(\lambda_{2}-\lambda_{2k})),
\end{align}
which implies
\begin{align}\label{eqn:conj}
	\overline{\chi_{k}(g_p g_q^{-1})}=\frac{|\chi_k(T)|}{\overline{\chi_k(T)}}\cdot \overline{\exp(-{\mathrm i}t(\lambda_{2}-\lambda_{2k}))}.
\end{align}
On the other hand,
\begin{align}\label{eqn:inverse-1}
	\overline{\chi_{k}(g_p g_q^{-1})}=\chi_{k^{-1}}(g_p g_q^{-1})&\overset{\eqref{eqn:f30}}{=}\frac{|\chi_{k^{-1}}(T)|}{\chi_{k^{-1}}(T)}\cdot \exp(-{\mathrm i}t(\lambda_{2}-\lambda_{2k^{-1}}))=\frac{|\chi_k(T)|}{\overline{\chi_k(T)}}\cdot \exp(-{\mathrm i}t(\lambda_{2}-\lambda_{2k^{-1}})).
\end{align}
Since
$$\lambda_{2k}=\chi_k(R)+|\chi_k(T)|=\overline{\chi_k(R)}+|\overline{\chi_k(T)}|
=\chi_{k^{-1}}(R)+|\overline{\chi_k(T)}|=\lambda_{2k^{-1}},$$
we have by \eqref{eqn:inverse-1} that
\begin{align}\label{eqn:conj-8-6}
	\overline{\chi_{k}(g_p g_q^{-1})}=\frac{|\chi_k(T)|}{\overline{\chi_k(T)}}\cdot \exp(-{\mathrm i}t(\lambda_{2}-\lambda_{2k})).
\end{align}
Combining \eqref{eqn:conj} and \eqref{eqn:conj-8-6}, we have
$$\overline{\exp(-{\mathrm i}t(\lambda_{2}-\lambda_{2k}))}=\exp(-{\mathrm i}t(\lambda_{2}-\lambda_{2k})),$$
which implies $\exp(-{\mathrm i}t(\lambda_{2}-\lambda_{2k}))=\pm 1$, and so
\begin{equation}\label{eqn:jj'}
t(\lambda_{2}-\lambda_{2k})\in \{z\pi\mid z\in \ZZ\}.
\end{equation}
Therefore, under the assumption of \eqref{eqn:H01t=1}, \eqref{eqn:H01i=l} holds if and only if
\begin{equation*}
	\chi_k(g_qg_q^{-1})=
	\left\{
	\begin{array}{ll}
		\frac{|\chi_k(T)|}{\chi_k(T)},~&\text{if }\exp(-{\mathrm i}t(\lambda_{2}-\lambda_{2k}))=1,\\
		-\frac{|\chi_k(T)|}{\chi_k(T)},~&\text{if }\exp(-{\mathrm i}t(\lambda_{2}-\lambda_{2k}))=-1,
	\end{array}
	\right.
\end{equation*}
and \eqref{eqn:jj'} holds. \end{proof}

Taking $T=\emptyset$ in Lemma \ref{PSTiff}, we can obtain Lemma 2.1 in \cite{f1}. When $T\neq \emptyset$, we have the following corollary.

\begin{core}\label{core-t}
Let $\Gamma=\mathrm{BiCay}(G;R,L,T)$ with $T\neq \emptyset$ be a bi-Cayley graph over a finite abelian group $G$. Then pairs of vertices having PST at the time $t$ in $\Gamma$ are either all from $(G_0\times G_0)\cup(G_1\times G_1)$, or all from $(G_0\times G_1)\cup (G_1\times G_0)$.
\end{core}
	
\begin{proof} Note that the two conditions $t(\lambda_{2}-\lambda_{1})\in\{2z\pi\mid z\in \ZZ\}$ and $t(\lambda_{2}-\lambda_{1})\in\{(2z+1)\pi\mid z\in \ZZ\}$ cannot hold at the same time $t$. Since $T\neq \emptyset$, we have $\chi_1(T)\neq 0$. Therefore, by Lemma \ref{PSTiff}, pairs of vertices having PST at the time $t$ in $\Gamma$ are either all from $(G_0\times G_0)\cup(G_1\times G_1)$, or all from $(G_0\times G_1)\cup (G_1\times G_0)$. \end{proof}
	
\begin{core}\label{core-8-11}
Let $\Gamma$ be a bi-Cayley graph over a finite abelian group $G$. If $\Gamma$ has PST between vertices $g_p$ and $g_q$ with $(g_p,g_q)\in (G_0\times G_0)\cup (G_1\times G_1)$ and $g_p\neq g_q$, then the order of $g_pg_q^{-1}$ is $2$, which implies that $|G|$ is even.
\end{core}

\begin{proof} By Lemma \ref{PSTiff}, when $(g_p,g_q)\in (G_0\times G_0)\cup (G_1\times G_1)$, $\chi_k(g_pg_q^{-1})=\pm 1$ for each $1\leq k\leq |G|$. Due to $\chi_{g_pg_q^{-1}}(g_k)=\chi_{k}(g_pg_q^{-1})$, we have $\chi_{g_pg_q^{-1}}(g_k)=\pm1$ for each $1\leq k\leq |G|$. Thus if $g_p\neq g_q$, $\chi_{g_pg_q^{-1}}$ is of order $2$. It follows that the order of $g_pg_q^{-1}$ is also $2$. \end{proof}

As a special case of Lemma 3.2 in \cite{zf1}, we can see that a bi-Cayley graph $\Gamma=\mathrm{BiCay}(G;R,L,T)$ with $R=L$ over an abelian group $G$ is isomorphic a Cayley graph over a generalized dihedral group. This yields the following corollary. Here we give a proof independent of  \cite[Lemma 3.2]{zf1}.

\begin{core}\label{core:GDn}
Let $\Gamma=\mathrm{BiCay}(G;R,L,T)$ be a bi-Cayley graph over a finite abelian group $G$. If $\Gamma$ has PST between vertices $g_p$ and $g_q$ with $(g_p,g_q)\in (G_0\times G_1)\cup (G_1\times G_0)$, then $\Gamma$ is a Cayley graph over a generalized dihedral group.
\end{core}

\begin{proof} By Lemma \ref{PSTiff}(3), if $\Gamma=\mathrm{BiCay}(G;R,L,T)$ over a finite abelian group $G$ has PST between vertices $g_p$ and $g_q$ with $(g_p,g_q)\in (G_0\times G_1)\cup (G_1\times G_0)$, then $R=L$. Let the vertex set of $\Gamma$ be the union of the right part $G_0=\{g_0\mid g\in G\}$ and the left part $G_1=\{g_1\mid g\in G\}$. Let $\tilde{G}=\left\langle G,b\mid b^2=1, bgb=g^{-1}, g\in G\right\rangle$ be the generalized dihedral group over $G$ and $S=R\cup bT\subseteq \tilde{G}$, where $bT=\{bt\mid t\in T\}$. Then $\mathrm{BiCay}(G;R,L,T)$ is isomorphic to $\mathrm{Cay}(\tilde{G},S)$ by using the vertex mapping $g_0\mapsto g$ and $g_1\mapsto bg$ for all $g\in G$. \end{proof}

\subsection{The case of $(g_p,g_q)\in (G_0\times G_1)\cup (G_1\times G_0)$}\label{sec:0-1}

Lemma \ref{PSTiff} provides a necessary and sufficient condition for an abelian bi-Caylay graph having PST. For ease of use, we shall refine Lemma \ref{PSTiff}. This subsection examines the case of $(g_p,g_q)\in (G_0\times G_1)\cup (G_1\times G_0)$.

We need the notation of the 2-adic exponential valuation of rational numbers which is a mapping defined by
$$v_2:\mathbb{Q}\longrightarrow \ZZ\cup \{\infty\},~v_2(0)=\infty,~v_2(2^l\frac{a}{b})=l,~where ~a,b,l\in \ZZ~and~2\nmid ab.$$
We assume that $\infty + \infty= \infty + l=\infty$ and $\infty\ge l$ for any $l\in \ZZ$.
Then $v_2$ has the following properties.
For $\beta,\beta'\in \mathbb{Q}$,
\begin{itemize}
	\item [(P1)] $v_2(\beta \beta')=v_2(\beta)+v_2(\beta')$;
	\item [(P2)] $v_2(\beta + \beta')\ge \min(v_2(\beta),v_2(\beta')), \mbox{ and the equality
		holds if }v_2(\beta)\neq v_2(\beta').$
\end{itemize}

A graph is said to be {\em integral} if the eigenvalues of its adjacency matrix are all integers. This is equivalent to saying that all eigenvalues are rational numbers, since they are algebraic integers.

\begin{lem}\label{lem:integral2}
	Let $\Gamma=\mathrm{BiCay}(G;R,L,T)$ be a bi-Cayley graph over an abelian group $G$ of order $n$.
	Assume that $\Gamma$ has PST between vertices $g_p$ and $g_q$ with $(g_p,g_q)\in (G_0\times G_1)\cup (G_1\times G_0)$ at the time $t$. Then
	\begin{itemize}
		\item [$1)$] $\Gamma$ is an integral graph;
		\item [$2)$] $\chi_k(T)\in \mathbb{Q}$ and $v_2(|\chi_k(T)|)=v_2(|T|)$ for each $1\leq k\leq n$.
	\end{itemize}	
\end{lem}

\begin{proof} 1) For each $1\leq k\leq n$, by Lemma \ref{PSTiff}, $t(\lambda_{2k}-\lambda_{2k-1})\in \{(2z+1)\pi\mid z\in \ZZ\}$. Let $t=\pi F$. Then
\begin{align}\label{eqn:10-9}
F(\lambda_{2k}-\lambda_{2k-1})\in \{2z+1\mid z\in \ZZ\}.
\end{align}
Since Lemma \ref{PSTiff} yields $R=L$, by \eqref{eqn:eigenvalues}, $\lambda_{2k}-\lambda_{2k-1}=2|\chi_k(T)|$, so
\begin{align}\label{eqn:10-9-1}
2F|\chi_k(T)|\in \{2z+1\mid z\in \ZZ\}.
\end{align}
Taking $k=1$ in \eqref{eqn:10-9-1}, we have $2F|T|\in \ZZ$, which implies $F\in \mathbb{Q}$. Furthermore, by Lemma \ref{PSTiff}, $t(\lambda_{2}-\lambda_{2k})\in \{z\pi\mid z\in\ZZ\}$ for each $1\leq k\leq n$. Due to $t=\pi F$, $F(\lambda_{2}-\lambda_{2{k}})\in \ZZ$. By \eqref{eqn:eigenvalues}, $\lambda_{2}=|R|+|T|\in \ZZ$, so $\lambda_{2{k}}\in \mathbb{Q}$, which yields $\lambda_{2{k}}\in \ZZ$. Then apply \eqref{eqn:10-9} to obtain $\lambda_{2{k}-1}\in \ZZ$. Therefore, $\Gamma$ is an integral graph.

2) Since $\lambda_{2k}-\lambda_{2k-1}=2|\chi_k(T)|\in \ZZ$, we have $|\chi_k(T)|\in \mathbb{Q}$. By \eqref{eqn:10-9-1}, $v_2(2F|\chi_k(T)|)=0$. It follows that $v_2(|\chi_k(T)|)=-1-v_2(F)$ for each $1\leq k\leq n$. Therefore, $v_2(|\chi_k(T)|)=v_2(|\chi_1(T)|)=v_2(|T|)$.
\end{proof}

For convenience, we set $2\ZZ=\{2z\mid z\in \ZZ\}$ and $2\ZZ+1=\{2z+1\mid z\in \ZZ\}$. The following lemma is simple but very useful.

\begin{lem}\label{lem:number theory}
Let $x$ be a rational number and let $d_1,d_2,\ldots,d_n$ be nonzero integers such that $\gcd(d_1,d_2,\cdots,d_n)=1$. Then
\begin{itemize}
	\item [$(1)$] $xd_k\in \ZZ$ for each $1\leq k\leq n$ if and only if $x\in \ZZ$;
	\item [$(2)$] $xd_k\in 2\ZZ$ for each $1\leq k\leq n$ if and only if $x\in 2\ZZ$;
	\item [$(3)$] $xd_k\in 2\ZZ+1$ for each $1\leq k\leq n$ if and only if $x\in 2\ZZ+1$ and $d_k\in 2\ZZ+1$ for each $1\leq k\leq n$.
\end{itemize}
\end{lem}

\begin{proof} The sufficiency is straightforward. It suffices to examine the necessity. Let $\alpha_\ZZ\in \{\ZZ,2\ZZ\}$. If $xd_k\in \alpha_\ZZ$ for each $1\leq k\leq n$, then
\begin{align}\label{eqn:1119}
	x\in \bigcap_{1\leq k\leq n}\frac{1}{d_k}\alpha_\ZZ
	&=\bigcap_{1\leq k\leq n} \frac{1}{d_1d_2\cdots d_n}d_1\cdots d_{k-1}d_{k+1}\cdots d_n \alpha_\ZZ\\
	&=\frac{1}{d_1d_2\cdots d_n}\cdot {\rm lcm}(d_1\cdots d_{k-1}d_{k+1}\cdots d_n\mid 1\leq k\leq n)\cdot\alpha_\ZZ\notag\\
	&=\frac{1}{\gcd(d_1,d_2,\cdots d_n)}\cdot\alpha_\ZZ
	=\alpha_\ZZ.\notag
\end{align}
If $xd_k\in 2\ZZ+1$ for each $1\leq k\leq n$, then $v_2(xd_k)=0$, and so $v_2(d_k)=-v_2(x)$, which implies that $v_2(d_k)=v_2(d_{k'})$ for any $k\neq k'$. Due to $\gcd(d_1,d_2,\cdots,d_n)=1$, we have $d_k\in 2\ZZ+1$ for each $1\leq k\leq n$. Thus $$\bigcap_{1\leq k\leq n} d_1\cdots d_{k-1}d_{k+1}\cdots d_n (2\ZZ+1)={\rm lcm}(d_1\cdots d_{k-1}d_{k+1}\cdots d_n\mid 1\leq k\leq n)\cdot(2\ZZ+1).$$
Using the the same argument as that in \eqref{eqn:1119}, we have $x\in 2\ZZ+1$.
\end{proof}

\begin{thm}\label{PSTiffG0G1}
Let $\Gamma=BiCay(G;R,L,T)$ be a bi-Cayley graph over an abelian group $G$ of order $n$. Then $\Gamma$ has PST between vertices $g_p$ and $g_q$ with $(g_p,g_q)\in (G_0\times G_1)\cup (G_1\times G_0)$ at the time $t$ if and only if the following conditions hold, where $\lambda_1,\lambda_2,\ldots,\lambda_{2n}$ are given in Lemma $\ref{lem:eigenvalues and eigenvectors of Dhat}$, $M=\gcd(\lambda_{2}-\lambda_{2{k}}\mid 1\leq k\leq n)$ and $M_T=\gcd(|\chi_k(T)|\mid 1\leq k\leq n)$.
\begin{itemize}
\item [$1)$] $\Gamma$ is an integral graph and $R=L$.		
\item [$2)$] For each $1\leq k\leq n$,
    \begin{itemize}
    \item [$2.1)$] $\chi_{k}({g_p}{g_q}^{-1})=\left\{
			\begin{array}{ll}
			\frac{|\chi_k(T)|}{\chi_k(T)}, & {\rm{if}}\ (g_p,g_q)\in G_0\times G_1\  {\rm{and}} \ \exp(\rm{i}t(\lambda_{2}-\lambda_{2{k}}))=1 ;\\[3mm]

            -\frac{|\chi_k(T)|}{\chi_k(T)}, & {\rm{if}}\ (g_p,g_q)\in G_0\times G_1\  {\rm{and}} \ \exp(\rm{i}t(\lambda_{2}-\lambda_{2{k}}))=-1 ;\\[3mm]

             \frac{|\chi_k(T)|}{\overline{\chi_k(T)}}, & {\rm{if}}\ (g_p,g_q)\in G_1\times G_0\  {\rm{and}} \ \exp(\rm{i}t(\lambda_{2}-\lambda_{2{k}}))=1 ;\\[3mm]

            -\frac{|\chi_k(T)|}{\overline{\chi_k(T)}}, & {\rm{if}}\ (g_p,g_q)\in  G_1\times G_0\  {\rm{and}} \ \exp(\rm{i}t(\lambda_{2}-\lambda_{2{k}}))=-1;
			\end{array}
			\right.$
    \item [$2.2)$] $\chi_k(T)\neq 0$;
    \item [$2.3)$] $v_2(|\chi_k(T)|)=v_2(|T|)$.
    \end{itemize}
\item [$3)$] If $M>0$, then $v_2(M)> v_2(|T|)$.
\item [$4)$] $t\in\{\frac{(1+2z)\pi}{\gcd(2M_T,M)}\mid z\in\ZZ\}$.
\end{itemize}
\end{thm}

\begin{proof} By Lemma \ref{PSTiff}, $\Gamma$ has PST between vertices $g_p$ and $g_q$ with $(g_p,g_q)\in (G_0\times G_1)\cup (G_1\times G_0)$ at the time $t$ if and only if $R=L$ and for each $1\leq k\leq n$,
\begin{itemize}
	\item [$i)$] $\chi_{k}({g_p}{g_q}^{-1})=\left\{
		\begin{array}{ll}
			\frac{|\chi_k(T)|}{\chi_k(T)}, & {\rm{if}}\ (g_p,g_q)\in G_0\times G_1\  {\rm{and}} \ \exp(\rm{i}t(\lambda_{2}-\lambda_{2{k}}))=1 ;\\[3mm]
			
			-\frac{|\chi_k(T)|}{\chi_k(T)}, & {\rm{if}}\ (g_p,g_q)\in G_0\times G_1\  {\rm{and}} \ \exp(\rm{i}t(\lambda_{2}-\lambda_{2{k}}))=-1 ;\\[3mm]
			
			\frac{|\chi_k(T)|}{\overline{\chi_k(T)}}, & {\rm{if}}\ (g_p,g_q)\in G_1\times G_0\  {\rm{and}} \ \exp(\rm{i}t(\lambda_{2}-\lambda_{2{k}}))=1 ;\\[3mm]
			
			-\frac{|\chi_k(T)|}{\overline{\chi_k(T)}}, & {\rm{if}}\ (g_p,g_q)\in  G_1\times G_0\  {\rm{and}} \ \exp(\rm{i}t(\lambda_{2}-\lambda_{2{k}}))=-1;
		\end{array}
		\right.$
	\item [$ii)$] $\chi_k(T)\neq 0$;
	\item [$iii)$] $t(\lambda_{2k}-\lambda_{2k-1})\in \{(2z+1)\pi\mid z\in \ZZ\}$;
	\item [$iv)$] $t(\lambda_{2}-\lambda_{2k})\in \{z\pi\mid z\in \ZZ\}$.
\end{itemize}

If $\Gamma$ has PST between vertices $g_p$ and $g_q$ with $(g_p,g_q)\in (G_0\times G_1)\cup (G_1\times G_0)$ at the time $t$, then Lemma \ref{PSTiff} yields Conditions $2.1)$ and $2.2)$, Lemmas \ref{PSTiff} and \ref{lem:integral2} yields Condition $1)$, and Lemma \ref{lem:integral2} yields Condition $2.3)$. Thus to complete the proof, it suffices to show, under the assumption of Conditions $1)$ and $2)$, the equivalence between Conditions $iii)$, $iv)$ and Conditions $3)$, $4)$.

Since $R=L$, by Lemma \ref{lem:eigenvalues and eigenvectors of Dhat}, for each $1\leq k\leq n$,  $\lambda_{2k}-\lambda_{2k-1}=2|\chi_k(T)|$. Since $\Gamma$ is an integral graph, $2|\chi_k(T)|\in \ZZ$. By Condition $2.3)$, $v_2(|\chi_k(T)|)=v_2(|T|)\geq 0$, and so $|\chi_k(T)|\in \ZZ$. By Condition $2.2)$, $\chi_k(T)\neq 0$ for each $1\leq k\leq n$, and we can define
$$M_T=\gcd(|\chi_k(T)|\mid 1\leq k\leq n),$$
which is a positive integer. We will show that, under the assumption of Conditions $1)$ and $2)$, Condition $iii)$ is equivalent to
\begin{equation}\label{eqn:10-10}
t\in \{\frac{\pi}{2M_T}(2z+1)\mid z\in \ZZ\}.
\end{equation}
Let $|\chi_k(T)|=M_Td_k$ for some positive integer $d_k$.
By Condition $iii)$, $t(\lambda_{2k}-\lambda_{2k-1})=t\cdot2|\chi_k(T)|=t\cdot2M_Td_k\in \{(2z+1)\pi\mid z\in \ZZ\}$.
Clearly $\gcd(d_1,d_2,\ldots,d_n)=1$, and Condition $2.3)$ implies that $d_k$ is odd for each $1\leq k\leq n$. Thus setting $t=\pi F$ and applying Lemma \ref{lem:number theory}, one can see that Condition $iii)$ is equivalent to \eqref{eqn:10-10}.

Now consider Condition $(iv)$, i.e., $t(\lambda_{2}-\lambda_{2k})\in \{z\pi\mid z\in \ZZ\}$. Due to $R=L$, by Lemma \ref{lem:eigenvalues and eigenvectors of Dhat}, $\lambda_{2}-\lambda_{2{k}}=|R|-\chi_k(R)+|T|-|\chi_k(T)|$. Since $\chi_k(R)\leq |R|$,  $|\chi_k(T)|\leq |T|$ and $\Gamma$ is an integral graph, $\lambda_{2}-\lambda_{2{k}}$ is a nonnegative integer. So we can define
$$
	M=\gcd(\lambda_{2}-\lambda_{2{k}}\mid 1\leq k\leq n).
$$

If $\lambda_{2}\neq \lambda_{2{k}}$ for some $1\leq k\leq n$, then $M$ is a positive integer. By a similar argument to that for \eqref{eqn:10-10}, one can see that Condition $iv)$ is equivalent to
\begin{equation}\label{eqn:10-10-1}
t\in \{\frac{\pi}{M}z\mid z\in \ZZ\}.
\end{equation}
Combining \eqref{eqn:10-10} and \eqref{eqn:10-10-1}, we have
\begin{equation}\label{eqn:TG0G1}
	t\in\frac{\pi}{2M_T}(2\ZZ+1)\cap \frac{\pi}{M}\ZZ=\frac{\pi}{2M_TM}\big(M(2\ZZ+1)\cap 2M_T\ZZ\big).
\end{equation}
For every $z\in \ZZ$,
\begin{center}
\begin{align*}
	&~z\in M(2\ZZ+1)\cap 2M_T\ZZ\\
	\Longleftrightarrow &~z=(2x+1)M=2yM_T \text{ for some }x,y\in\ZZ\\
	\Longleftrightarrow &~2yM_T-2xM=M\\
	\Longleftrightarrow &~\gcd(2M_T,2M)\mid M\\
    \Longleftrightarrow &~v_2(M)> v_2(M_T)\geq 0. 	
\end{align*}
\end{center}
Since $M_T=\gcd(|\chi_k(T)|\mid 1\leq k\leq n)$, and by Condition 2.3), $v_2(|\chi_k(T)|)=v_2(|T|)$ for any $1\leq k\leq n$, we have $v_2(M_T)=v_2(|T|)$. Therefore,
$$v_2(M)> v_2(M_T)\geq 0 \Longleftrightarrow v_2(M)>v_2(|T|),$$
which yields Condition 3). On the other hand, $v_2(M)>v_2(M_T)$ implies $v_2(M)\geq 1+v_2(M_T)=v_2(2M_T)$. This implies that
$$M(2\ZZ+1)\cap 2M_T\ZZ={\rm lcm}(M,2M_T)\cdot (2\ZZ+1).$$
Hence by \eqref{eqn:TG0G1}, we have $$t\in\frac{\pi}{2M_TM} \cdot {\rm lcm}(M,2M_T)\cdot (2\ZZ+1)= \frac{\pi}{\gcd(2M_T,M)}(2\ZZ+1),$$
which yields Condition 4).
To summarize, if $\lambda_{2}\neq \lambda_{2{k}}$ for some $1\leq k\leq n$, then under the assumption of Conditions $1)$ and $2)$, Conditions $iii)$ and $iv)$ are equivalent to Conditions $3)$ and $4)$.

If $\lambda_{2}= \lambda_{2{k}}$ for every $1\leq k\leq n$, then Condition $iv)$ always holds and $M=0$. So $\frac{\pi}{\gcd(2M_T,M)}(2\ZZ+1)=\frac{\pi}{2M_T}(2\ZZ+1)$, which coincides with \eqref{eqn:10-10}. Therefore, if $\lambda_{2}= \lambda_{2{k}}$ for every $1\leq k\leq n$, then under the assumption of Conditions $1)$ and $2)$, Conditions $iii)$ and $iv)$ are equivalent to Conditions $3)$ and $4)$. \end{proof}

\subsection{The case of $(g_p,g_q)\in (G_0\times G_0)\cup (G_1\times G_1)$}\label{sec:0-0}

\begin{lem}\label{lem:integer0} {\rm \cite[Theorem 3.1]{PeriodicGraphs}}
A graph is periodic if and only if the ratio of its any two nonzero eigenvalues is rational.
\end{lem}

\begin{lem}\label{lem:period=integral}
Let $\Gamma=\mathrm{BiCay}(G;R, L, T)$ be a bi-Cayley graph over a finite abelian group $G$.
Then $\Gamma$ is periodic if and only if $\Gamma$ is an integral graph.
\end{lem}

\begin{proof} If $\Gamma$ is integral, then $\Gamma$ is periodic by Lemma \ref{lem:integer0}. Suppose that $\Gamma$ is periodic. By Lemma \ref{lem:integer0}, if we can show that $\Gamma$ has a nonzero rational eigenvalue, then all eigenvalues of $\Gamma$ are rational, and so $\Gamma$ is an integral graph.

If $T=\emptyset$, then $R\cup L\neq \emptyset$. By Lemma \ref{lem:eigenvalues and eigenvectors of Dhat}, $\{\lambda_1,\lambda_2\}=\{|L|,|R|\}$, and so at least one of $\lambda_1$ and $\lambda_2$ is a positive integer.

If $R=L=\emptyset$, then $T\neq \emptyset$. By Lemma \ref{lem:eigenvalues and eigenvectors of Dhat},  eigenvalues of $\Gamma$ are $\pm |\chi_k(T)|$, $1\leq k\leq |G|$, where $|\chi_1(T)|=|T|$ is a positive integer.

If $T\neq \emptyset$ and $R\cup L\neq \emptyset$, then by Lemma \ref{lem:eigenvalues and eigenvectors of Dhat}, $\lambda_{2}\neq 0$. If $\lambda_{1}=0$, then $\sqrt{(|R|-|L|)^2+4|T|^2}=|R|+|L|$, and so $\lambda_{2}=\frac{|R|+|L|+\sqrt{(|R|-|L|)^2+4|T|^2}}{2}=|R|+|L|$, which is a positive integer. If $\lambda_{1}\neq 0$, then by Lemma \ref{lem:integer0}, $\frac{\lambda_{1}}{\lambda_{2}}=\frac{|R|+|L|-\sqrt{(|R|-|L|)^2+4|T|^2}}{|R|+|L|+\sqrt{(|R|-|L|)^2+4|T|^2}}
=\frac{x}{y},$ where $x$ and $y$ are nonzero integers. Then
$\sqrt{(|R|-|L|)^2+4|T|^2}=\frac{(y-x)(|R|+|L|)}{x+y}\in \mathbb{Q}$. Thus $\lambda_{2}=\frac{|R|+|L|+\sqrt{(|R|-|L|)^2+4|T|^2}}{2}$ is a nonzero rational eigenvalue.
\end{proof}

\begin{lem}\label{lem:PSTiffG0G0}
Let $\Gamma=BiCay(G;R,L,T)$ be an integral bi-Cayley graph over a finite abelian group $G$. Assume that $\Gamma$ has PST between vertices $g_p$ and $g_q$ with $(g_p,g_q)\in (G_0\times G_0)\cup (G_1\times G_1)$ at the time $t$. For $j\in\{-1,1\}$, let $\Omega_{j}=\{k\mid  \chi_k(g_pg_q^{-1})=j, 1\leq k\leq n\}.$
Let $H=\{k\mid \chi_k(T)=0, 1\leq k\leq n\}.$ Write $X=R$ if $(g_p,g_q)\in G_0\times G_0$, and $X=L$ if $(g_p,g_q)\in G_1\times G_1$.
Then there exists an integer $\mu$ such that
\begin{itemize}
    \item [$1)$] $v_2(\lambda_{2k}-\lambda_{2k-1})\geq \mu+1$ for all $k\not\in H$;
	\item [$2)$] $v_2(\lambda_{2}-\lambda_{2k-1})=\mu$ for all $k\in \Omega_{-1}\setminus H$, and $v_2(\lambda_{2}-\lambda_{2k-1})\geq \mu+1$ for all $k\in \Omega_{1}\setminus H$;
	\item [$3)$] if $T\neq\emptyset$, then $v_2(\lambda_{2}-\chi_k(X))=\mu$ for all $k\in \Omega_{-1} \cap H$, and $v_2(\lambda_{2}-\chi_k(X))\geq \mu+1$ for all $k\in \Omega_{1} \cap H$;
    \item [$4)$] if $T=\emptyset$, then $v_2(|X|-\chi_k(X))=\mu$ for all $k\in \Omega_{-1}$, and $v_2(|X|-\chi_k(X))\geq\mu+1$ for all $k\in \Omega_{1}$.
\end{itemize}
\end{lem}

\begin{proof} If $\Gamma$ has PST between vertices $g_p$ and $g_q$ with $(g_p,g_q)\in (G_0\times G_0)\cup (G_1\times G_1)$ at the time $t$, then by Lemma \ref{PSTiff}, $\chi_k(g_pg_q^{-1})=\pm 1$ for $1\leq k\leq n$, and so $\Omega_{-1}\cup \Omega_{1}=\{1,2,\ldots,n\}$. Set $t=2\pi F$.

1) If $\chi_k(T)\neq 0$, i.e. $k\not\in H$, then by Lemma \ref{PSTiff}, $F(\lambda_{2k}-\lambda_{2k-1})\in \ZZ$. Thus $v_2(F(\lambda_{2k}-\lambda_{2k-1}))\geq 0$. Since $\Gamma$ is an integral graph, we have $F\in \mathbb{Q}$. Hence $v_2(\lambda_{2k}-\lambda_{2k-1})\geq -v_2(F)$.

2) If $\chi_k(T)\neq 0$ and $\chi_{k}({g_p}{g_q}^{-1})=-1$, i.e. $k\in \Omega_{-1}\setminus H$, then by Lemma \ref{PSTiff}, $F(\lambda_{2}-\lambda_{2k-1})\in \{\frac{1}{2}+z\mid z\in \ZZ\}$. Thus $v_2(F(\lambda_{2}-\lambda_{2k-1}))=-1$, which implies $v_2(\lambda_{2}-\lambda_{2k-1})=-1-v_2(F)$.
If $\chi_k(T)\neq 0$ and $\chi_{k}({g_p}{g_q}^{-1})=1$, i.e., $k\in \Omega_{1}\setminus H$, then $F(\lambda_{2}-\lambda_{2k-1})\in \ZZ$. Hence $v_2(\lambda_{2}-\lambda_{2k-1})\geq -v_2(F)$.

3) Assume that $T\neq\emptyset$. If $\chi_k(T)=0$ and $\chi_{k}({g_p}{g_q}^{-1})=-1$, i.e., $k\in \Omega_{-1}\cap H$, then $F(\lambda_{2}-\chi_k(X))\in \{\frac{1}{2}+z\mid z\in \ZZ\}$. Hence $v_2(\lambda_{2}-\chi_k(X))=-1-v_2(F)$. If $\chi_k(T)=0$ and $\chi_{k}({g_p}{g_q}^{-1})=1$, i.e., $k\in \Omega_{1}\cap H$, then $F(\lambda_{2}-\chi_k(X))\in \ZZ$. Hence $v_2(\lambda_{2}-\chi_k(X))\geq -v_2(F)$.

4) Assume that $T=\emptyset$. If $k\in \Omega_{-1}$, then $F(|X|-\chi_k(X))\in \{\frac{1}{2}+z\mid z\in \ZZ\}$. Hence $v_2(|X|-\chi_k(X))=-1-v_2(F)$. If $k\in \Omega_{1}$, then $F(|X|-\chi_k(X))\in \ZZ$. Hence $v_2(|X|-\chi_k(X))\geq -v_2(F)$.

Now take $\mu=-1-v_2(F)$ to complete the proof.
\end{proof}

\begin{thm}\label{PSTiffG0G0}
Let $\Gamma=BiCay(G;R,L,T)$ be an integral bi-Cayley graph over an abelian group $G$ of order $n$. Write
$$H=\{k\mid \chi_k(T)=0, 1\leq k\leq n\}.$$
Let $\lambda_1,\lambda_2,\ldots,\lambda_{2n}$ be as in Lemma $\ref{lem:eigenvalues and eigenvectors of Dhat}$. Let
$$M_{0}=\gcd(\lambda_{2k}-\lambda_{2k-1}\mid 1\leq k\leq n, k\notin H)$$
and
$$M_{1}=\gcd(\lambda_{2}-\lambda_{2k-1}\mid 1\leq k\leq n, k\notin H).$$
Then $\Gamma$ has PST between vertices $g_p$ and $g_q$ with $(g_p,g_q)\in (G_0\times G_0)\cup (G_1\times G_1)$ at the time $t$ if and only if the following conditions hold, where for $j\in\{-1,1\}$,
$$\Omega_{j}=\{k\mid  \chi_k(g_pg_q^{-1})=j, 1\leq k\leq n\};$$
$X=R$ if $(g_p,g_q)\in G_0\times G_0$ and $X=L$ if $(g_p,g_q)\in G_1\times G_1$;
$$M_{\emptyset,X}=\gcd(|X|-\chi_k(X)\mid 1\leq k\leq n)$$
and
$$	M_{X}=\left\{
		\begin{array}{ll}
			\gcd(\lambda_{2}-\chi_k(X)\mid k\in H), & {\rm{if\ }}\ H\neq \emptyset;\\[3mm]
			0, & {\rm{if\ }}\ H=\emptyset.
		\end{array}
		\right.$$
\begin{itemize}
		\item [$1)$] For each $1\leq k\leq n$, $\chi_{k}({g_p}{g_q}^{-1})=\pm 1$.
		\item [$2)$] There exists an integer $\mu$ such that
        \begin{itemize}
        \item [$2.1)$] $\left\{
		\begin{array}{ll}
		v_2(\lambda_{2}-\lambda_{2k-1})=\mu & {\rm{for\ all }}\  k\in \Omega_{-1}\setminus H;\\[3mm]
		v_2(\lambda_{2}-\chi_k(X))=\mu & {\rm{for }}\ T\neq\emptyset\ {\rm{and\ all}}\ k\in \Omega_{-1} \cap H;\\[3mm]
        v_2(|X|-\chi_k(X))=\mu & {\rm{for }}\ T=\emptyset\ {\rm{and\ all}}\ k\in \Omega_{-1};\\[3mm]
		\end{array}
		\right.$
        \item [$2.2)$]
        $\left\{
		\begin{array}{ll}
		v_2(\lambda_{2k}-\lambda_{2k-1})\geq \mu+1 & {\rm{for\ all }}\  k\not\in H;\\[3mm]
		v_2(\lambda_{2}-\lambda_{2k-1})\geq \mu+1 & {\rm{for\ all }}\ k\in \Omega_{1}\setminus H;\\[3mm]
v_2(\lambda_{2}-\chi_k(X))\geq \mu+1 & {\rm{for }}\ T\neq\emptyset\ {\rm{and\ all}}\ k\in \Omega_{1} \cap H;\\[3mm]
v_2(|X|-\chi_k(X))\geq\mu+1 & {\rm{for }}\ T=\emptyset\ {\rm{and\ all}}\ k\in \Omega_{1}.\\[3mm]
		\end{array}
		\right.$
        \end{itemize}
		\item [$3)$] If $T\neq \emptyset$, then
                    $$t\in\left\{
		              \begin{array}{ll}
           \{\frac{(1+2z)\pi}{\gcd(M_0,M_{1},M_{X})}\mid z\in\ZZ\},& {\rm{when\ }} \Omega_{-1}\neq\emptyset; \\[3mm]			
			\{\frac{2\pi z}{\gcd(M_0,M_{1},M_{X})}\mid z\in\ZZ\}, & {\rm{when\ \Omega_{-1}=\emptyset}}.
		\end{array}
		      \right.$$
		\item [$4)$] If $T=\emptyset$, then
		$$t\in \left\{
		\begin{array}{ll}
            \{\frac{2\pi z}{M_{\emptyset,X}}\mid z\in \ZZ\}, & {\rm{when\ }} \Omega_{-1}=\emptyset {\rm{\ or\ for\ each }}\ k\in \Omega_{-1},\ \chi_k(X)=|X|;\\[3mm]
            \{\frac{(1+2z)\pi}{M_{\emptyset,X}}\mid z\in \ZZ\}, & {\rm{otherwise}}.
		\end{array}
		\right.$$
	\end{itemize}	
\end{thm}

\begin{proof} By Lemma \ref{PSTiff}, $\Gamma$ has PST between vertices $g_p$ and $g_q$ with $(g_p,g_q)\in (G_0\times G_0)\cup (G_1\times G_1)$ at the time $t$ if and only if for each $1\leq k\leq n$,
\begin{itemize}
	\item [$i)$] $\chi_{k}({g_p}{g_q}^{-1})=\pm 1$;
	\item [$ii)$] if $k\not\in H$, then $t(\lambda_{2k}-\lambda_{2k-1})\in \{2z\pi\mid z\in \ZZ\}$ and $t(\lambda_{2}-\lambda_{2k-1})\in \{(2z-\frac{\chi_{k}({g_p}{g_q}^{-1})-1}{2})\pi\mid z\in \ZZ\}$;
    \item [$iii)$] if $T\neq \emptyset$ and $k\in H$, then $t(\lambda_{2}-\chi_k(X))\in \{(2z-\frac{\chi_{k}({g_p}{g_q}^{-1})-1}{2})\pi\mid z\in \ZZ\}$;
	\item [$iv)$] if $T=\emptyset$, then
	 $t(|X|-\chi_k(X))\in\{(2z-\frac{\chi_{k}({g_p}{g_q}^{-1})-1}{2})\pi\mid z\in \ZZ\}$.
\end{itemize}

If $\Gamma$ has PST between vertices $g_p$ and $g_q$ with $(g_p,g_q)\in (G_0\times G_0)\cup (G_1\times G_1)$ at the time $t$, then Lemma \ref{PSTiff} yields Condition $1)$, and Lemma \ref{lem:PSTiffG0G0} yields Condition $2)$.
Thus to complete the proof, it suffices to show, in the assumption of Conditions $1)$ and $2)$, the equivalence between Conditions $ii)$, $iii)$, $iv)$ and Conditions $3)$, $4)$.

If $T=\emptyset$, then $H=\{1,2,\ldots,n\}$ and so there does not exist $1\leq k\leq n$ such that $k\not\in H$. It follows that Condition $ii)$ works only if $T\neq\emptyset$.
Condition $iii)$, too, works only if $T\neq\emptyset$. This leads us to divide the proof into two parts according to whether $T$ is nonempty.

\textbf{Case I}: $T\neq \emptyset$. Clearly $1\in \Omega_1\setminus H$. We can define
\begin{align*}\label{eqn:M0}
M_0=\gcd(\lambda_{2k}-\lambda_{2k-1}\mid 1\leq k\leq n, k\notin H),
\end{align*}
and
\begin{equation*}\label{eqn:M11}
M_{1,1}=\gcd(\lambda_{2}-\lambda_{2k-1}\mid k\in \Omega_1\setminus H).
\end{equation*}
If $\Omega_{-1}\setminus H\neq \emptyset$, define
\begin{equation*}\label{eqn:M1-1}
M_{1,-1}=\gcd(\lambda_{2}-\lambda_{2k-1}\mid k\in \Omega_{-1}\setminus H).
\end{equation*}
By Lemma \ref{lem:eigenvalues and eigenvectors of Dhat}, one can check that $M_0$, $M_{1,1}$ and $M_{1,-1}$ are positive integers.
We will show that, in the assumption of Conditions $1)$ and $2)$, if $\Omega_{-1}\setminus H\neq \emptyset$, then Condition $ii)$ is equivalent to
\begin{align}
	&tM_0\in \{2z\pi \mid z\in \ZZ\},\label{eqn:tM0}\\
	&tM_{1,1}\in \{2z\pi \mid z\in \ZZ\},\label{eqn:tM11}\\
	&tM_{1,-1}\in \{(2z+1)\pi\mid z\in \ZZ \},\label{eqn:tM1-1}
\end{align}
and if $\Omega_{-1}\setminus H=\emptyset$, then Condition $ii)$ is equivalent to \eqref{eqn:tM0} and \eqref{eqn:tM11}.
For $1\leq k\leq n$ and $k\notin H$, let $\lambda_{2k}-\lambda_{2k-1}=M_0d_k$ for some integer $d_k$. Clearly $\gcd(d_1,d_2,\cdots,d_n)=1$.
Thus setting $t=\pi F$ and applying Lemma \ref{lem:number theory}, we have that $t(\lambda_{2k}-\lambda_{2k-1})\in \{2z\pi\mid z\in \ZZ\}$ if and only if \eqref{eqn:tM0} holds.
Similarly, for $k\in \Omega_1\setminus H$, $t(\lambda_{2}-\lambda_{2k-1})\in \{2z\pi\mid z\in \ZZ\}$ if and only if \eqref{eqn:tM11} holds.
Therefore, when $\Omega_{-1}\setminus H= \emptyset$, Condition $ii)$ is equivalent to \eqref{eqn:tM0} and \eqref{eqn:tM11}.
If $\Omega_{-1}\setminus H\neq \emptyset$, let $\lambda_{2}-\lambda_{2k-1}=M_{1,-1}d'_{k}$ for some integer $d'_{k}$. By Condition $2.1)$, $v_2(\lambda_{2}-\lambda_{2k-1})=\mu$ for every $k\in \Omega_{-1}\setminus H$, so $d'_{k}$ is odd for every $k\in \Omega_{-1}\setminus H$.
Applying Lemma \ref{lem:number theory} again, we have that $t(\lambda_{2}-\lambda_{2k-1})\in \{(2z+1)\pi\mid z\in \ZZ\}$ if and only if \eqref{eqn:tM1-1} holds. Therefore, when $\Omega_{-1}\setminus H\neq \emptyset$, Condition $ii)$ is equivalent to \eqref{eqn:tM0}, \eqref{eqn:tM11} and \eqref{eqn:tM1-1}.
For convenience, set
\begin{align*}\label{eqn:M1}
	M_{1}=\gcd(\lambda_{2}-\lambda_{2k-1}\mid 1\leq k\leq n, k\notin H).
\end{align*}
Then $M_1=\gcd(M_{1,1},M_{1,-1})$ if $\Omega_{-1}\setminus H\neq \emptyset$, and $M_1=M_{1,1}$ if $\Omega_{-1}\setminus H=\emptyset$.

Consider Condition $iii)$.
If $H=\emptyset$, Condition $iii)$ is trivial. Assume that $H\neq \emptyset$. Since $\Omega_1\cup \Omega_{-1}=\{1,2,\ldots,n\}$, either $H\cap \Omega_1\neq \emptyset$ or $H\cap \Omega_{-1}\neq \emptyset$. For $j\in\{1,-1\}$, when $H\cap \Omega_j\neq \emptyset$, we can define
\begin{equation*}\label{eqn:MX1}
	M_{X,j}=\gcd(\lambda_{2}-\chi_k(X)\mid k\in H\cap \Omega_j).
\end{equation*}
By Lemma \ref{lem:eigenvalues and eigenvectors of Dhat}, $\lambda_{2}-\chi_k(X)\geq \frac{|R|+|L|+\sqrt{(|R|-|L|)^2+4|T|^2}}{2}-|X|>0$ (since $T\neq \emptyset$), so $M_{X,j}$ is a positive integer. By similar arguments to those for Condition $ii)$, one can see that in the assumption of Conditions $1)$ and $2)$, if $H\cap \Omega_1=\emptyset$, Condition $iii)$ is equivalent to
\begin{equation}\label{eqn:tMX-1}
	tM_{X,-1}\in \{(2z+1)\pi\mid z\in \ZZ \};
\end{equation}
if $H\cap \Omega_{-1}=\emptyset$, Condition $iii)$ is equivalent to
\begin{equation}\label{eqn:tMX1}
tM_{X,1}\in \{2z\pi \mid z\in \ZZ\};
\end{equation}
and if $H\cap \Omega_{j}\neq\emptyset$ for every $j\in \{-1,1\}$, Condition $iii)$ is equivalent to \eqref{eqn:tMX-1} and \eqref{eqn:tMX1}. For convenience, set
\begin{equation*}\label{eqn:MX}
	M_{X}=\left\{
		\begin{array}{ll}
			\gcd(\lambda_{2}-\chi_k(X)\mid k\in H), & {\rm{if\ }}\ H\neq \emptyset;\\[3mm]
			0, & {\rm{if\ }}\ H= \emptyset.
		\end{array}
		\right.
\end{equation*}
Then $M_X=M_{X,-1}$ if $H\neq \emptyset$ and $H\cap \Omega_{1}=\emptyset$; $M_X=M_{X,1}$ if $H\neq \emptyset$ and $H\cap \Omega_{-1}=\emptyset$; $M_X=\gcd(M_{X,1},M_{X,-1})$ if $H\neq \emptyset$, $H\cap \Omega_{1}\neq\emptyset$ and $H\cap \Omega_{-1}\neq\emptyset$.

Now combine Conditions $ii)$ and $iii)$.

If $H=\emptyset$ and $\Omega_{-1}\setminus H\neq \emptyset$, then by \eqref{eqn:tM0}, \eqref{eqn:tM11} and \eqref{eqn:tM1-1},
\begin{equation}\label{eqn:t11}
\begin{split}
	t\in& \frac{2\pi}{M_0}\ZZ\cap \frac{2\pi}{M_{1,1}}\ZZ\cap \frac{\pi}{M_{1,-1}}(1+2\ZZ) =\frac{2\pi}{M_0}\ZZ\cap \frac{\pi}{M_{1,1}M_{1,-1}}(2M_{1,-1}\ZZ\cap M_{1,1}(1+2\ZZ)).
\end{split}
\end{equation}
Consider the set $2M_{1,-1}\ZZ\cap M_{1,1}(1+2\ZZ)$. We have
\begin{equation}\label{eqn:1zzzzzzz}
	\begin{split}
		&~z\in 2M_{1,-1}\ZZ\cap M_{1,1}(1+2\ZZ)\\
		\Longleftrightarrow&~z=2M_{1,-1}z_1=M_{1,1}(1+2z_2) \text{ for some }z_1,z_2\in\ZZ\\
		\Longleftrightarrow &~2z_1M_{1,-1}-2z_2M_{1,1}=M_{1,1}\\
		\Longleftrightarrow &~\gcd(2M_{1,-1},2M_{1,1})\mid M_{1,1}\\
		\Longleftrightarrow &~v_2(M_{1,1})\geq \mu +1~(\text{note that } v_2(M_{1,-1})=\mu \text{ by Condition 2.1))}.
	\end{split}
\end{equation}
This coincides with Condition $2.2)$.
Furthermore, since $v_2(M_{1,1})\geq \mu +1$ and $v_2(M_{1,-1})=\mu$, we have
\begin{equation}\label{eqn:1ttttttt00}
	\begin{split}
		2M_{1,-1}\ZZ\cap M_{1,1}(1+2\ZZ)
		= &{\rm lcm}(2M_{1,-1},M_{1,1})(1+2\ZZ)
		=\frac{2M_{1,-1}M_{1,1}}{\gcd(2M_{1,-1},M_{1,1})}(1+2\ZZ).
	\end{split}
\end{equation}
Since $M_{1,-1}$ is a positive integer, $v_2(M_{1,-1})=\mu\geq 0$, and so $v_2(M_{1,1})\geq \mu+1\geq 1$, which implies that $M_{1,1}$ is even. Thus applying $v_2(M_{1,1})\geq \mu +1$ and $v_2(M_{1,-1})=\mu$ again, we have
\begin{equation}\label{eqn:1ttttttt000000000}
	\begin{split}		
\frac{2M_{1,-1}M_{1,1}}{\gcd(2M_{1,-1},M_{1,1})}(1+2\ZZ)=
\frac{M_{1,-1}M_{1,1}}{\gcd(M_{1,-1},M_{1,1})}(1+2\ZZ)=\frac{M_{1,-1}M_{1,1}}{M_1}(1+2\ZZ).
	\end{split}
\end{equation}
Thus combining \eqref{eqn:t11} and \eqref{eqn:1ttttttt000000000}, we have
\begin{equation}\label{eqn:12ttttttt}
	\begin{split}
t\in \frac{2\pi}{M_0}\ZZ\cap \frac{\pi}{M_1}(1+2\ZZ)=\frac{\pi}{M_0M_1}(2M_1\ZZ\cap M_0(1+2\ZZ)).
	\end{split}	
\end{equation}
For the set $2M_1\ZZ\cap M_0(1+2\ZZ)$, similar argument shows that
$2M_1\ZZ\cap M_0(1+2\ZZ)\neq \emptyset \Longleftrightarrow v_2(M_0)\ge \mu+1$ (due to $v_2(M_1)=\mu$), which coincides with Condition $2.2)$. Furthermore, $2M_1\ZZ\cap M_0(1+2\ZZ)=\frac{M_0M_1}{\gcd(M_0,M_1)}(1+2\ZZ)$.
Then \eqref{eqn:12ttttttt} becomes
\begin{equation}\label{eqn:tM1}
		t\in \frac{\pi}{\gcd(M_0,M_1)}(1+2\ZZ).
\end{equation}
Note that $M_{X}=0$ when $H=\emptyset$. Hence \eqref{eqn:tM1} is equivalent to $t\in \frac{\pi}{\gcd(M_0,M_1,M_X)}(1+2\ZZ)$, which yields Condition $3)$.

If $H=\emptyset$ and $\Omega_{-1}\setminus H=\emptyset$, then by \eqref{eqn:tM0} and \eqref{eqn:tM11},
\begin{equation}\label{eqn:t-12-3-1}
	t\in\frac{2\pi}{M_0}\ZZ\cap \frac{2\pi}{M_{1,1}}\ZZ.
\end{equation}
If $H\neq \emptyset$, $\Omega_{-1}\setminus H=\emptyset$ and $H\cap \Omega_{-1}=\emptyset$, then by \eqref{eqn:tM0}, \eqref{eqn:tM11} and \eqref{eqn:tMX1},
\begin{equation}\label{eqn:t-12-3-2}
	t\in \frac{2\pi}{M_0}\ZZ\cap \frac{2\pi}{M_{1,1}}\ZZ\cap \frac{2\pi}{M_{X,1}}\ZZ.
\end{equation}
By similar arguments to those for \eqref{eqn:t11}, one can check that \eqref{eqn:t-12-3-1} (resp. \eqref{eqn:t-12-3-2}) holds if and only if $t\in \frac{2\pi}{\gcd(M_0,M_{1},M_X)}\ZZ$, which yields Condition 3). Note that $\Omega_{-1}=\emptyset$ if and only if $H=\emptyset$ and $\Omega_{-1}\setminus H=\emptyset$, or $H\neq \emptyset$, $\Omega_{-1}\setminus H=\emptyset$ and $H\cap \Omega_{-1}=\emptyset$.

If $H\neq \emptyset$, $\Omega_{-1}\setminus H\neq \emptyset$ and $H\cap \Omega_{1}=\emptyset$,
then by \eqref{eqn:tMX-1}, \eqref{eqn:tM0}, \eqref{eqn:tM11} and \eqref{eqn:tM1-1},
\begin{equation}\label{eqn:t-3}
	t\in \frac{\pi}{M_{X,-1}}(1+2\ZZ)\cap \frac{2\pi}{M_0}\ZZ\cap \frac{2\pi}{M_{1,1}}\ZZ\cap \frac{\pi}{M_{1,-1}}(1+2\ZZ).
\end{equation}
If $H\neq \emptyset$, $\Omega_{-1}\setminus H=\emptyset$ and $H\cap \Omega_{1}=\emptyset$, then by \eqref{eqn:tM0}, \eqref{eqn:tM11} and \eqref{eqn:tMX-1},
\begin{equation}\label{eqn:t-2}
	t\in \frac{2\pi}{M_0}\ZZ\cap \frac{2\pi}{M_{1,1}}\ZZ\cap \frac{\pi}{M_{X,-1}}(1+2\ZZ).
\end{equation}
One can check that \eqref{eqn:t-3} (resp. \eqref{eqn:t-2}) holds if and only if $t\in \frac{\pi}{\gcd(M_0,M_{1},M_{X})}(1+2\ZZ)$, which yields Condition $3)$.

If $H\neq \emptyset$, $\Omega_{-1}\setminus H \neq \emptyset$ and $H\cap \Omega_{-1}=\emptyset$, then by \eqref{eqn:tM0}, \eqref{eqn:tM11}, \eqref{eqn:tM1-1} and \eqref{eqn:tMX1},
\begin{equation}\label{eqn:t-4}
t\in \frac{2\pi}{M_0}\ZZ\cap \frac{2\pi}{M_{1,1}}\ZZ\cap \frac{\pi}{M_{1,-1}}(1+2\ZZ)\cap \frac{2\pi}{M_{X,1}}\ZZ.
\end{equation}
If $H\neq \emptyset$, $\Omega_{-1}\setminus H\neq \emptyset$, $H\cap \Omega_{1}\neq \emptyset$ and $H\cap \Omega_{-1}\neq \emptyset$,
then by \eqref{eqn:tM0}, \eqref{eqn:tM11}, \eqref{eqn:tM1-1}, \eqref{eqn:tMX-1} and \eqref{eqn:tMX1},
\begin{equation}\label{eqn:t-5}
	t\in \frac{2\pi}{M_0}\ZZ\cap \frac{2\pi}{M_{1,1}}\ZZ\cap \frac{\pi}{M_{1,-1}}(1+2\ZZ)\cap\frac{2\pi}{M_{X,1}}\ZZ\cap \frac{\pi}{M_{X,-1}}(1+2\ZZ).
\end{equation}
If $H\neq \emptyset$, $\Omega_{-1}\setminus H=\emptyset$, $H\cap \Omega_{1}\neq \emptyset$ and $H\cap \Omega_{-1}\neq \emptyset$,
then by \eqref{eqn:tM0}, \eqref{eqn:tM11}, \eqref{eqn:tMX-1} and \eqref{eqn:tMX1},
\begin{equation}\label{eqn:t-6}
	t\in \frac{2\pi}{M_0}\ZZ\cap \frac{2\pi}{M_{1,1}}\ZZ\cap\frac{2\pi}{M_{X,1}}\ZZ\cap \frac{\pi}{M_{X,-1}}(1+2\ZZ).
\end{equation}
One can check that \eqref{eqn:t-4} (resp. \eqref{eqn:t-5}, \eqref{eqn:t-6}) holds if and only if $t\in \frac{\pi}{\gcd(M_0,M_{1},M_{X})}(1+2\ZZ)$, which yields Condition $3)$.

\noindent\textbf{Case II}: $T=\emptyset$. Consider Condition $iv)$, i.e., $t(|X|-\chi_k(X))\in\{(2z-\frac{\chi_{k}({g_p}{g_q}^{-1})-1}{2})\pi\mid z\in \ZZ\}$. Let
\begin{equation}\label{eqn:MEX}
	M_{\emptyset,X}=\gcd(|X|-\chi_k(X)\mid 1\leq k\leq n).
\end{equation}
Clearly $|X|\geq \chi_k(X)$ for any $k$. If for each $1\leq k\leq n$, $\chi_k(X)=|X|$, then the adjacency matrix of $Cay(G,X)$ is the diagonal matrix $|X|I$. Since $X$ does not contain the identity element of $G$, a contradiction occurs. Therefore, $M_{\emptyset,X}$ is a positive integer. Furthermore, for $j\in \{-1,1\}$, if $\Omega_j\neq \emptyset$, define
\begin{equation}\label{eqn:MEXj}
M_{\emptyset,X,j}=\gcd(|X|-\chi_k(X)\mid k\in \Omega_j).
\end{equation}
Since $M_{\emptyset,X}$ is a positive integer, at least one of $M_{\emptyset,X,1}$ or $M_{\emptyset,X,-1}$ is a positive integer.

Clearly $\Omega_1\neq\emptyset$. If $\chi_k(X)=|X|$ for any $k\in \Omega_1$, then Condition $iv)$ is equivalent to $t(|X|-\chi_k(X))\in\{(2z+1)\pi\mid z\in \ZZ\}$ for any $k\in \Omega_{-1}$. Applying Lemma \ref{lem:number theory} together with the use of Condition 2.1), we have
$tM_{\emptyset,X,-1}\in \{(2z+1)\pi\mid z\in \ZZ\}$. Note that $M_{\emptyset,X}=M_{\emptyset,X,-1}$ if $\chi_k(X)=|X|$ for any $k\in \Omega_1$. Therefore, $t\in \{\frac{(2z+1)\pi}{M_{\emptyset,X}}\mid z\in \ZZ\}$, which yields Condition $4)$.

If $\Omega_{-1}=\emptyset$ or $\chi_k(X)=|X|$ for any $k\in \Omega_{-1}$, then Condition $iv)$ is equivalent to $t(|X|-\chi_k(X))\in\{2z\pi\mid z\in \ZZ\}$ for any $k\in \Omega_{1}$. Applying Lemma \ref{lem:number theory}, we have
$tM_{\emptyset,X,1}\in \{2z\pi\mid z\in \ZZ\}$. Note that $M_{\emptyset,X}=M_{\emptyset,X,1}$ if $\Omega_{-1}=\emptyset$ or $\chi_k(X)=|X|$ for any $k\in \Omega_{-1}$. Therefore, $t\in \{\frac{2z\pi}{M_{\emptyset,X}}\mid z\in \ZZ\}$, which yields Condition $4)$.

If for every $j\in\{-1,1\}$, $\Omega_j\neq \emptyset$ and there exists $k\in \Omega_j$ such that $\chi_k(X)\neq |X|$, then Condition $iv)$ is equivalent to $t(|X|-\chi_k(X))\in \{2z\pi\mid z\in \ZZ\}$ for any $k\in \Omega_{1}$ and $t(|X|-\chi_{k}(X))\in \{(2z+1)\pi\mid z\in \ZZ\}$ for any $k\in \Omega_{-1}$. Applying Lemma \ref{lem:number theory} again, we have $tM_{\emptyset,X,1}\in \{2z\pi\mid z\in \ZZ\}$ and $tM_{\emptyset,X,-1}\in \{(2z+1)\pi\mid z\in \ZZ\}$. Thus
\begin{equation}\label{eqn:tMEX1,-1,2}
t\in \frac{\pi}{M_{\emptyset,X,1}}2\ZZ\cap \frac{\pi}{M_{\emptyset,X,-1}}(1+2\ZZ).
\end{equation}
By similar arguments to those for \eqref{eqn:t11}, one can check that \eqref{eqn:tMEX1,-1,2} holds if and only if $t\in \frac{\pi}{M_{\emptyset,X}}(1+2\ZZ)$ (note that $v_2(M_{\emptyset,X,1})\ge \mu +1$ by Condition $2.2)$), which yields Condition $4)$.
\end{proof}

\begin{rem}\label{remark}
\begin{itemize}
    \item [$1)$] In Theorem $\ref{PSTiffG0G0}$, if $\Gamma=\mathrm{BiCay}(G;R,L,T)$ has PST between vertices $g_p$ and $g_q$ with $(g_p,g_q)\in (G_0\times G_0)\cup (G_1\times G_1)$, then Condition $1)$ requires $\chi_{k}({g_p}{g_q}^{-1})=\pm 1$ for any $1\leq k\leq n$. This yields that $\chi_{{g_p}{g_q}^{-1}}(g_k)=\pm 1$ for any $1\leq k\leq n$. If $g_p\neq g_q$, then there exists $k$ such that $\chi_{{g_p}{g_q}^{-1}}(g_k)=-1$, and so $g_pg_q^{-1}$ is of order $2$. Therefore, when $g_p\neq g_q$, Condition $1)$ in Theorem $\ref{PSTiffG0G0}$ can be replaced by  the order of $g_pg_q^{-1}$ being $2$.
	\item [$2)$] If $\Omega_{-1}=\emptyset$, then Condition $2)$ in Theorem $\ref{PSTiffG0G0}$ is trivial. This is because Condition $2.1)$ does not work in this case and Condition $2.2)$ always holds because of the existence of $\mu$.
\end{itemize}
\end{rem}

\begin{rem}\label{remark1-4}
\begin{itemize}
    \item [$1)$] In Theorem $\ref{PSTiffG0G0}$, if $T=\emptyset$, then $\Gamma=\mathrm{BiCay}(G;R,L,T)$ is the union of two disjoint Cayley graphs. As corollaries, we can obtain Theorems $2.3$ and $2.4$ in $\cite{f1}$. Note that Theorems $2.3$ and $2.4$ in $\cite{f1}$ require the given Cayley graph is connected. The connectivity ensures that if $|X|=\chi_k(X)$ with $X\in\{R,L\}$ then $k=1$ $($this is because $|X|=\chi_k(X)$ implies $\chi_k(g)=1$ for any $g\in X$; due to the connectivity, $G=\langle X\rangle$, and so $\chi_k(g)=1$ for any $g\in G$, which yields $k=1)$. Therefore, in $4)$ of Theorem $\ref{PSTiffG0G0}$, the condition ``for each $k\in \Omega_{-1}$, $\chi_k(X)=|X|$'' does not work in the assumption of the connectivity.
    \item [$2)$] Theorem $\ref{PSTiffG0G0}$ does not require the connectivity of $\Gamma$. The reader can compare the discussion before Theorem $2.3$ in $\cite{f1}$ and the discussion after \eqref{eqn:MEX} in this paper for the reason.
\end{itemize}
\end{rem}

Theorem \ref{PSTiffG0G0} works only if the bi-Cayley graph $\Gamma$ is integral. It is not clear whether $\Gamma$ must be integral if $\Gamma$ has PST between vertices $g_p$ and $g_q$ with $(g_p,g_q)\in (G_0\times G_0)\cup (G_1\times G_1)$. However, we can provide a sufficient condition in the following lemma to make sure a bi-Cayley graph is integral. To this end, we introduce a new type of bi-Cayley graphs. A bi-Cayley graph $\Gamma=\mathrm{BiCay}(G;R,L,T)$ over an abelian group $G$ of order $n$ is said to be {\em weakly inner-cospectral} if $\{\chi_k(R)\mid \chi_k(T)=0,1\leq k\leq n\}=\{\chi_k(L)\mid \chi_k(T)=0,1\leq k\leq n\}$. Clearly any $\mathrm{BiCay}(G;R,L,T)$ over an abelian group $G$ with $T=\emptyset$ or $R=L$ is weakly inner-cospectral.


\begin{lem}\label{lem:integeral}
Let $\Gamma=\mathrm{BiCay}(G;R, L, T)$ be a weakly inner-cospectral bi-Cayley graph over a finite abelian group $G$. If $\Gamma$ has PST between vertices $g_p$ and $g_q$ with $(g_p,g_q)\in (G_0\times G_0)\cup (G_1\times G_1)$ at the time $t$, then $\Gamma$ is an integral graph.
\end{lem}

\begin{proof} Assume that $\Gamma$ has PST between $g_p$ and $g_q$ with $(g_p,g_q)\in G_0\times G_0$ at the time $t$ (a similar argument applies to $(g_p,g_q)\in G_1\times G_1$). Then by Lemma \ref{PSTiff}, $\chi_{k}({g_p}{g_q}^{-1})=\pm 1$ and
$$
\left\{\begin{array}{ll}
t(\lambda_{2k}-\lambda_{2k-1})\in\{2z\pi\mid z\in \ZZ\}\ \ {\rm{and}} &  \\
			\ \ \ t(\lambda_{2}-\lambda_{2k-1})\in \{(2z-\frac{\chi_{k}({g_p}{g_q}^{-1})-1}{2})\pi\mid z\in \ZZ\}, & {\rm{if}}\ \chi_k(T)\neq 0; \\	[3mm]
			t(|R|-\chi_k(R))\in\{(2z-\frac{\chi_{k}({g_p}{g_q}^{-1})-1}{2})\pi\mid z\in \ZZ\}, & {\rm{if}}\ T=\emptyset;\\	[3mm] t(\lambda_{2}-\chi_k(R))\in\{(2z-\frac{\chi_{k}({g_p}{g_q}^{-1})-1}{2})\pi\mid z\in \ZZ\}, & {\rm{if}}\ \chi_k(T)= 0, T\neq\emptyset.
\end{array}
\right.
$$	
Since $\Gamma$ is weakly inner-cospectral, we have
$$
\left\{\begin{array}{ll}
2t(\lambda_{2k}-\lambda_{2k-1})\in\{2z\pi\mid z\in \ZZ\}\ \ {\rm{and}}
			\ \ 2t(\lambda_{2}-\lambda_{2k-1})\in \{2z\pi\mid z\in \ZZ\}, & {\rm{if}}\ \chi_k(T)\neq 0; \\	[3mm]
			2t(|R|-\chi_k(R))\in\{2z\pi\mid z\in \ZZ\}\ \ {\rm{and}}
			\ \ 2t(|L|-\chi_k(L))\in\{2z\pi\mid z\in \ZZ\}\, & {\rm{if}}\ T=\emptyset;\\	[3mm] 2t(\lambda_{2}-\chi_k(R))\in\{2z\pi\mid z\in \ZZ\}\ \ {\rm{and}}
			\ \ 2t(\lambda_{2}-\chi_k(L))\in\{2z\pi\mid z\in \ZZ\}, & {\rm{if}}\ \chi_k(T)= 0, T\neq\emptyset.
\end{array}
\right.
$$
Note that when $T=\emptyset$, $\chi_k(T)=0$ for any $1\leq k\leq n$, so the weakly inner-cospectral property guarantees $\{\chi_k(R),1\leq k\leq n\}=\{\chi_k(L),1\leq k\leq n\}$, which implies $\chi_1(R)=|R|=|L|=\chi_1(L)$. Therefore, $\Gamma$ is periodic at the time $2t$. By Lemma \ref{lem:period=integral}, $\Gamma$ is integral.
\end{proof}

As a straightforward corollary of Lemma \ref{lem:integeral}, we have the following result.

\begin{core}\label{cor:integeral R=L}
Let $\Gamma=\mathrm{BiCay}(G;R,L,T)$ be a bi-Cayley graph over a finite abelian group $G$ with $R=L$ or $T=\emptyset$. If $\Gamma$ has PST between vertices $g_p$ and $g_q$ with $(g_p,g_q)\in (G_0\times G_0)\cup (G_1\times G_1)$ at the time $t$, then $\Gamma$ is an integral graph.
\end{core}

\subsection{Periodicity}
	
\begin{thm}\label{thm:periodTneqvarnothing}
Let $\Gamma=\mathrm{BiCay}(G;R,L,T)$ be an integral bi-Cayley graph over an abelian group $G$ of order $n$. Write
$$H=\{k\mid \chi_k(T)=0, 1\leq k\leq n\}.$$
Let $\lambda_1,\lambda_2,\ldots,\lambda_{2n}$ be as in Lemma $\ref{lem:eigenvalues and eigenvectors of Dhat}$.  Then $\Gamma$ is periodic for a vertex $g_p$ at the time
$$t\in \left\{
		\begin{array}{ll}
            \{\frac{2\pi z}{M_{\emptyset,X}}\mid z\in \ZZ\setminus \{0\}\} & {\rm{for\ }} T=\emptyset;\\[3mm]
            \{\frac{2\pi z}{\gcd(M_0,M_{1},M_{X})}\mid z\in\ZZ\setminus \{0\}\} & {\rm{for\ T\neq\emptyset}},
		\end{array}
		\right.$$
where $X=R$ if $g_p\in G_0$ and $X=L$ if $g_p\in G_1$,
$$M_{0}=\gcd(\lambda_{2k}-\lambda_{2k-1}\mid 1\leq k\leq n, k\notin H),$$ $$M_{1}=\gcd(\lambda_{2}-\lambda_{2k-1}\mid 1\leq k\leq n, k\notin H),$$
$$M_{X}=\gcd(\lambda_{2}-\chi_k(X)\mid k\in H),$$
and
$$M_{\emptyset,X}=\gcd(|X|-\chi_k(X)\mid 1\leq k\leq n).$$
Furthermore, a bi-Caylay graph $\Gamma=\mathrm{BiCay}(G;R,L,T)$ over an abelian group $G$ is periodic at the time $t$ if and only if $\Gamma$ is an integral graph and
$$t\in \left\{
		\begin{array}{ll}
            \{\frac{2\pi z}{\gcd(M_{\emptyset,R},M_{\emptyset,L})}\mid z\in \ZZ\setminus \{0\}\} & {\rm{for\ }} T=\emptyset;\\[3mm]
            \{\frac{2\pi z}{\gcd(M_0,M_1,M_R,M_L)}\mid z\in\ZZ\setminus \{0\}\} & {\rm{for\ T\neq\emptyset}}.
		\end{array}
		\right.$$
\end{thm}

\proof If $\Gamma$ is periodic for a vertex $g_p$, i.e., $\Gamma$ has PST between vertices $g_p$ and $g_p$, then $\Omega_{-1}$ defined in Theorem \ref{PSTiffG0G0} is an empty set. By Remark \ref{remark}(2), Condition $2)$ in Theorem $\ref{PSTiffG0G0}$ is trivial. By Lemma \ref{lem:period=integral}, a bi-Caylay graph over an abelian group is periodic if and only if it is an integral graph. Then it is readily checked that the desired result holds (note that when $T\neq\emptyset$,  we have $H\neq \emptyset$).  \qed

	
\section{Perfect state transfer on certain Cayley graphs}

If a Cayley graph $\Gamma$ can be seen as a bi-Cayley graph over an abelian group, then Theorem \ref{PSTiffG0G1} and Theorem \ref{PSTiffG0G0} provide necessary and sufficient conditions for $\Gamma$ having PST. This motivates us to examine PST on Cayley graphs that are also bi-Cayley graphs in this section.

An {\em extension} of a group $N$ by a group $F$ is a group $\tilde{G}$ that has a normal subgroup $G\cong N$ such that $\tilde{G}/G\cong F$. Let $N$ be a finite abelian group and let $\tilde{G}$ be an extension of $N$ by the cyclic group $\ZZ_2$. Then $\tilde{G}$ has a normal subgroup $G\cong N$ such that $\tilde{G}/G\cong \ZZ_2$. We can assume that $\tilde{G}=\left\langle bG \right\rangle=\{g,bg\mid g\in G\}$ with $b^2\in G$. When $b^2=1$, by establishing a relation between a Cayley graph over $\tilde{G}$ and a bi-Cayley graph over $G$, the following theorem gives a non-existence result on PST over $\tilde{G}$. Roughly speaking, a Cayley graph over $\tilde{G}$ can be viewed as a bi-Cayley graph over $G$ whose right part of the vertex set is $G$ and the left part is $bG$.

\begin{thm}\label{thm:n-even}
Let $N$ be an abelian group of order $n$. Let $\tilde{G}$ be an extension of $N$ by $\ZZ_2$ such that $\tilde{G}$ has a subgroup $G\cong N$ and $\tilde{G}=G:\left\langle b\right\rangle$ with $b^2=1$. Let
$\Gamma=\mathrm{Cay}(\tilde{G},\tilde{S})$ be a Cayley graph satisfying that
\begin{itemize}
\item [$1)$] $1\notin \tilde{S}=\tilde{S}^{-1}$;
\item [$2)$] for $g\in G$, $bg\in \tilde{S}$ if and only if $gb\in \tilde{S}$.
\end{itemize}
If $n\equiv 1\pmod{2}$, then $\Gamma$ has no PST between any pair of distinct vertices.
\end{thm}

\proof Let $R=\tilde{S}\cap G$, $L=b\tilde{S}b\cap G$ and $T=b\tilde{S}\cap G$. Due to $1\not\in\tilde{S}=\tilde{S}^{-1}$, we have $R=R^{-1}$, $L=L^{-1}$ and $1\notin R\cup L$. Construct a bi-Cayley graph $\mathrm{BiCay}(G;R,L,T)$ whose vertex set is the union of the right part $G_0=\{g_0\mid g\in G\}$ and the left part $G_1=\{g_1\mid g\in G\}$. It is readily checked that $\mathrm{Cay}(\tilde{G},\tilde{S})$ is isomorphic to $\mathrm{BiCay}(G;R,L,T)$ by using the vertex mapping $g\mapsto g_0$ and $bg\mapsto g_1$ for all $g\in G$.

Furthermore, if $T\neq \emptyset$, then for any $t\in T$, $t$ can be written as $b\cdot bg=g$ for some $g\in G$ and $bg\in \tilde{S}$. By the assumption, $bg\in \tilde{S}$ yields $gb\in \tilde{S}$. Since $\tilde{S}=\tilde{S}^{-1}$, we have $(gb)^{-1}=bg^{-1}\in \tilde{S}$, and so $b\cdot bg^{-1}=g^{-1}=t^{-1}\in T$. Thus $T=T^{-1}$, which implies that $\chi (T)$ is a real number for any $\chi\in \hat{G}$.

If $\mathrm{BiCay}(G;R,L,T)$ has PST between distinct vertices $g_p$ and $g_q$, then by Theorems \ref{PSTiffG0G1} and \ref{PSTiffG0G0}, $\chi_{k}({g_p}{g_q}^{-1})=\pm 1$ for all $1\leq k\leq n$. Therefore, $\chi_{g_pg_q^{-1}}$ is of order $2$, and so $n\equiv 0\mod 2$. \qed

\begin{rem}\label{remark:2022-1-13}
\begin{itemize}
\item [$1)$] In Theorem $\ref{thm:n-even}$, if $bgb=g$ for all $g\in G$, then $\tilde{G}$ is an abelian group. As a corollary of Theorem $\ref{thm:n-even}$, we can obtain Theorem $3.5$ in $\cite{f1}$.
\item [$2)$] In Theorem $\ref{thm:n-even}$, if $G=\ZZ_n$ and $bgb=g^{-1}$ for all $g\in G$, then $\tilde{G}$ is a dihedral group. As a corollary of Theorem $\ref{thm:n-even}$, we can obtain the the first half of Theorem $3.1$ in $\cite{C.d1}$. 
\end{itemize}
\end{rem}

Next, we give an example to determine perfect state transfers for an infinite family of Cayley graphs over dihedral groups by applying the relationship between Cayley graphs and bi-Cayley graphs established in the proof of Theorem \ref{thm:n-even}.

\begin{example}\label{ex:non-normalCay}
Let $m$ be a positive integer and $\tilde{G}=D_{8m}=\left\langle a,b \mid a^{8m}=b^2=1, ab=ba^{-1}\right\rangle$ be a dihedral group of order $16m$.
Let $\tilde{S}=\{a^{2j-1}\mid 1\leq j\leq 4m\}\cup \{ba^{2m},ba^{6m}\}$ and $\Gamma=\mathrm{Cay}(D_{8m},\tilde{S})$.
Then for any $1\leq i\leq 8m$ and $j\in\{0,1\}$, $\Gamma$ has PST between vertices $b^{j}a^i$ and $b^{j}a^{i+4m}$ at any time $t\in \{\frac{(1+2z)\pi}{2}\mid z\in \ZZ\}$.
Moreover, $\Gamma$ is periodic at any time $t\in \{z\pi\mid z\in\ZZ\}$.
\end{example}

\proof $\tilde{G}$ has a normal subgroup $G=\left\langle a\right\rangle $ and $\tilde{G}=\left\langle bG\right\rangle $ with $b^2=1$ and $bab=a^{-1}$.
Let $R=\tilde{S}\cap G=\{a^{2j-1}\mid 1\leq j\leq 4m\}$,
$L=b\tilde{S}b\cap G=\{a^{2j-1}\mid 1\leq j\leq 4m\}$ and $T=b\tilde{S}\cap G=\{a^{2m},a^{6m}\}$.
Construct a bi-Cayley graph $\mathrm{BiCay}(G;R,L,T)$ whose vertex set is the union of the right part $G_0=\{a^i_0\mid a^i\in G\}$ and the left part $G_1=\{a^i_1\mid a^i\in G\}$.
Then $\mathrm{Cay}(\tilde{G},\tilde{S})$ is isomorphic to $\mathrm{BiCay}(G;R,L,T)$ by using the vertex mapping $a^i\mapsto a^i_0$ and $ba^i\mapsto a^i_1$ for all $a^i\in G$.

First we apply Theorem \ref{thm:eigenvalues and eigenvectors of D} to calculate all eigenvalues of $\mathrm{BiCay}(G;R,L,T)$.
Since $G$ is an abelian group, $\hat{G}\cong G$.
Write $\hat{G}=\{\chi_k\mid 1\leq k\leq 8m\}$ and the mapping
$$\chi_k:G\longrightarrow \mathbb{C},~~\chi_{k}(a^i)=\omega_{8m}^{(k-1)i}$$
is a character of $G$, where $\omega_{8m}$ is a primitive $8m$-th root of unity in $\mathbb{C}$.
Simple calculation shows that for $1\leq k\leq 8m$,
$$\chi_k(R)=\chi_k(L)=\left\{
\begin{array}{ll}
	4m, &\ k=1;\\
	-4m, &\ k=4m+1;\\
	0, &\ \text{otherwise},	
\end{array}
\right.$$
and
$$\chi_k(T)=\left\{
\begin{array}{ll}
0, &\ k\equiv 0\pmod{4};\\
2, &\ k\equiv 1\pmod{4};\\
0, &\ k\equiv 2\pmod{4};\\
-2, &\ k\equiv 3\pmod{4}.	
\end{array}
\right.$$
By Theorem \ref{thm:eigenvalues and eigenvectors of D}, the eigenvalues of $\mathrm{BiCay}(G;R,L,T)$ are $\lambda_{2{k}-1}=\chi_k(R)-|\chi_k(T)|$ and $\lambda_{2{k}}=\chi_k(R)+|\chi_k(T)|$, $1\leq k\leq 8m$. They are calculated in the following table
\begin{center} 
	\begin{tabular}{c|cccccc}
		\toprule
		$k$&$1$&$4m+1$& even & otherwise\\ \hline
		$\lambda_{2k-1}$&$4m-2$&$-4m-2$&$0$&$-2$\\
		$\lambda_{2k}$&$4m+2$&$-4m+2$&$0$&$2$\\
		\bottomrule
	\end{tabular}
\end{center}
and we can see that $\Gamma$ is an integral graph. Note that $H=\{k\mid \chi_k(T)=0,1\leq k\leq 8m\}=\{2z\mid 1\leq z\leq 4m\}\neq \emptyset$. Therefore, by Theorem \ref{PSTiffG0G1}, $\mathrm{BiCay}(G;R,L,T)$ cannot have PST between any pair of vertices from $(G_0\times G_1) \cup (G_1\times G_0)$.

Now we consider the vertices $a^i$ and $a^{i'}$ with $(a^i,a^{i'})\in (G_0\times G_0) \cup (G_1\times G_1)$ for some  $1\leq i,i'\leq 8m$. Since $a^{4m}$ is the unique element of order $2$ in $G$, by Condition $1)$ in Theorem \ref{PSTiffG0G0} and Remark \ref{remark}, if $\mathrm{BiCay}(G;R,L,T)$ has PST between vertices $a^i$ and $a^{i'}$ with $i\neq i'$, then $a^{i'}=a^{i+4m}$.
To determine whether there exists PST between $a^i$ and $a^{i+4m}$ at time $t$, it suffices to examine Conditions 2) and 3) in Theorem \ref{PSTiffG0G0}. Clearly, the set $\Omega_{1}=\{k\mid \chi_k(a^{4m})=1,1\leq k\leq 8m\}=\{2z-1\mid 1\leq z\leq 4m\}$ and $\Omega_{-1}=\{k\mid \chi_k(a^{4m})=-1,1\leq k\leq 8m\}=\{2z\mid 1\leq z\leq 4m\}$.
For each $k\in \Omega_{-1}\cap H=H$, $v_2(\lambda_{2}-\chi_k(R))=v_2(\lambda_{2}-\chi_k(L))=v_2(4m+2)=1$, denoted by $\mu$.
For $k\notin H$, $v_2(\lambda_{2{k}}-\lambda_{2{k}-1})=v_2(4)=2=\mu +1$.
For $k\in \Omega_{1}\setminus H$, $v_2(\lambda_{2}-\lambda_{2{k}-1})\geq 2=\mu+1$.
Thus Condition 2) of Theorem \ref{PSTiffG0G0} holds. Furthermore, $M_0=\gcd(\lambda_{2{k}}-\lambda_{2{k}-1}\mid 1\leq k\leq 8m,k\notin H)=4$, $M_1=\gcd(\lambda_{2}-\lambda_{2{k}-1}\mid 1\leq k\leq 8m,k\notin H)=4$ and  $M_R=M_L=\gcd(\lambda_{2}-\chi_k(R)\mid k\in H)=4m+2$. Then by Condition $3)$ of Theorem \ref{PSTiffG0G0}, $t\in\{\frac{(1+2z)\pi}{2}\mid z\in \ZZ\}$. Therefore, $\mathrm{BiCay}(G;R,L,T)$ has PST between $a^i$ and $a^{i+4m}$ for any $(a^i,a^{i+4m})\in (G_0\times G_0) \cup (G_1\times G_1)$ at any time $t\in \{\frac{(1+2z)\pi}{2}\mid z\in \ZZ\}$.
That is to say, for any $1\leq i\leq 8m$ and $j\in\{0,1\}$, $\mathrm{Cay}(\tilde{G},\tilde{S})$ has PST between vertices $b^{j}a^i$ and $b^{j}a^{i+4m}$ at any time $t\in \{\frac{(1+2z)\pi}{2}\mid z\in \ZZ\}$.
Moreover, by Theorem \ref{thm:periodTneqvarnothing}, $\Gamma$ is periodic at any time $t\in \{z\pi\mid z\in\ZZ\}$. \qed

\begin{rem}\label{rem:conterexample}
A Cayley graph $\mathrm{Cay}(G,S)$ is called {\em normal} if $S$ is closed under conjugation. It was shown in $\cite{C.d2}$ that if a connected Cayley graph $\mathrm{Cay}(D_n,S)$ with $n$ even has PST between two distinct vertices, then $S$ is normal. However, the graph $\mathrm{Cay}(D_{8m},\tilde{S})$ in Example $\ref{ex:non-normalCay}$ is a connected non-normal Cayley graph, which gives a counterexample of Theorem $11$ in $\cite{C.d2}$. The correctness of Example $\ref{ex:non-normalCay}$ can be also checked by applying Corollary $\ref{core:H-entry}$ directly.
\end{rem}

In Theorem $\ref{thm:n-even}$, if $G$ is an abelian group and $bgb=g^{-1}$ for all $g\in G$, then $\tilde{G}$ is a generalized dihedral group. By applying Theorems \ref{PSTiffG0G1}, \ref{PSTiffG0G0} and \ref{thm:periodTneqvarnothing} together with the use of the relationship between Cayley graphs and bi-Cayley graphs established in the proof of Theorem \ref{thm:n-even}, one can characterize PST on the bi-Cayley graphs over generalized dihedral groups in the following example. We do not provide details here.

\begin{example}
Let $m$ be a positive integer and $G=\left\langle a, h\mid a^{8m}=h^2=1, ah=ha\right\rangle$. Let $\tilde{G}=\left\langle G,b\mid b^2=1, bgb=g^{-1}, g\in G\right\rangle$ be the generalized dihedral group over $G$. Let $\tilde{S}=\{a^{2j-1}\mid 1\leq j\leq 4m\}\cup\{a^{4m}h,b\}$ and $\Gamma=\mathrm{Cay}(\tilde{G},\tilde{S})$.
Then $\Gamma$ has PST between vertices $x$ and $y$ at any time $t\in \{(\frac{1}{2}+z)\pi\mid z\in \ZZ\}$ for any $x\in G$, $y\in bG$ and $y\neq bx$, and $\Gamma$ is periodic at any time $t\in \{z\pi\mid z\in\ZZ\}$.
\end{example}

In the concluding section of \cite{C.d1}, an example of Cayley graphs over generalized quaternion groups is mentioned without proof since \cite{C.d1} only focuses on examining PST on Cayley graphs over dihedral groups. We list this example here. The interested reader can check it from the point of view of bi-Cayley graphs. Note that this example concerns about the case of $\tilde{G}=\left\langle bG\right\rangle$ with $1\neq b^2\in G$.

\begin{example}\label{eg:Q} {\rm \cite[Lemma 6.1]{C.d1}}
Let $Q_{4m}=\left\langle a,b\mid a^{2m}=1,b^2 =a^m,b^{-1}ab=a^{-1}\right\rangle$ be a generalized quaternion group of order $4m$ with $m\equiv 0\mod{2}$. Let  $S=\{a^{\frac{m}{2}},a^{\frac{3m}{2}}\}\cup b\left\langle a\right\rangle$ and $\Gamma=\mathrm{Cay}(Q_{4m},S)$. Then for any $u\in Q_{4m}$, $\Gamma$ has PST between $u$ and $ua^m$ at any time $t\in \{\frac{(1+2z)\pi}{2}\mid z\in \ZZ\}$, and $\Gamma$ is periodic at any time $t\in \{\pi z\mid z\in\ZZ\}$.
\end{example}

\proof Let $G=\left\langle a\right\rangle $, $R=L=\{a^{\frac{m}{2}},a^{\frac{3m}{2}}\}$ and $T=G$. Construct a bi-Cayley graph $\mathrm{BiCay}(G;R,L$, $T)$ whose vertex set is the union of the right part $G_0=\{a^i_0\mid a^i\in G\}$ and the left part $G_1=\{a^i_1\mid a^i\in G\}$. Then $\mathrm{Cay}(Q_{4m},S)$ is isomorphic to $\mathrm{BiCay}(G;R,L,T)$ by using the vertex mapping $a^i\mapsto a^i_0$ and $ba^i\mapsto a^i_1$ for all $a^i\in G$. Apply Theorems \ref{PSTiffG0G1}, \ref{PSTiffG0G0} and \ref{thm:periodTneqvarnothing} to complete the proof. \qed

\section{Concluding remarks}

In this paper, we give some necessary and sufficient conditions to determine whether a bi-Cayley graph over a finite abelian group has PST. We show that if a bi-Cayley graph over abelian group has PST between two vertices from different vertex parts, it must be an integral graph (see Theorems \ref{PSTiffG0G1}) and is isomorphic to a Cayley graph over a generalized dihedral group (see Corollary \ref{core:GDn}). It is shown that a bi-Cayley graph over an abelian group is periodic if and only if it is an integral graph (see Lemma \ref{lem:period=integral}).
If there exists PST between two vertices belonging to a same part of the vertex set, then the number of vertices of this graph must be a multiple of four (see Corollary \ref{core-8-11}). For a bi-Cayley graph $\rm{BiCay}(G;R,L,T)$ with $T\neq \emptyset$, if there exists PST in some pairs of vertices at the same time $t$, then these pairs of vertices are either all from the same vertex part or all from different vertex parts (see Corollary \ref{core-t}). Furthermore, we show that, if the number of vertices is not a multiple of four, then there is no PST for a certain class of bi-Cayley graphs which are also Cayley graphs (see Theorem \ref{thm:n-even}). Especially, we give an example of a connected non-normal Cayley graph having PST between two distinct vertices (see Example \ref{ex:non-normalCay}), which produces a counterexample of Theorem $11$ in $\cite{C.d2}$. Comparing with the work in \cite{C.d1,C.d2,LCWW,f1}, we do not require the given bi-Cayley graph is connected (see Remark \ref{remark}).

Cao and Feng \cite{C.d1} pointed out that their method can be used to study PST in Cayley graphs over generalized dihedral groups $GD_n$ and quaternion groups $Q_{4n}$, but did not provide details. Since Cayley graphs over $GD_n$ and $Q_{4n}$ are also bi-Cayley graphs over abelian groups, our Theorems \ref{PSTiffG0G1} and \ref{PSTiffG0G0} can produce necessary and sufficient conditions to determine whether these graphs have PST. We only make use of representations of abelian groups, and do not need to analyze representations of $GD_n$ and $Q_{4n}$. In \cite{LCWW}, Luo et al. studied PST of Cayley graphs over semi-dihedral groups. Cayley graphs over semi-dihedral groups can be seen as bi-Cayley graphs over cyclic groups. Our results are also effective to deal with such kind of graphs.

The integrality is critical to a bi-Cayley graph such that it has PST. There have been a lot of research on the characterization of Cayley integral graphs. However, it is still an open problem to give a practical characterization for integral bi-Cayley graphs. Other problems related to this paper that are worth studying are listed below.

\begin{itemize}
\item Is every bi-Cayley graph over an abelian group having PST between a pair of vertices from the same vertex part integral?
\item Does there exist a connected irregular bi-Cayley graph over an abelian group having PST?
\item Does there exist a bi-Cayley graph over an abelian group having PST with $T\neq T^{-1}$?
\item Give examples of bi-Cayley graphs having PST over abelian groups which are not Cayley graphs.
\item Consider other types of state transfer on bi-Cayley graphs, such as pretty good state transfer (cf. \cite{gkklm,kly}), edge state transfer (cf. \cite{cg}), Laplacian perfect state transfer (cf. \cite{cl}) and state transfer on oriented graphs (cf. \cite{gl}).
\end{itemize}		

\subsection*{Acknowledgements}

The authors thank the anonymous referees for their valuable comments and suggestions that helped improve the equality of the paper. Special thanks goes to one of the reviewers for leading us to know the reference \cite{Arezoomand}.

\end{document}